\documentclass[11pt]{amsart}

\usepackage{amsfonts, amscd,color,amsmath,amssymb,amsthm}
\usepackage{color, graphicx}

\usepackage{tikz}

 \DeclareMathOperator{\diam}{diam}

 \DeclareMathOperator{\cl}{cl}

 \DeclareMathOperator{\dist}{dist}

\newcommand{\J}{\mathbb{J}}
\newcommand{\I}{\mathbb{I}}
\newcommand{\A}{\mathcal{A}}
\newcommand{\R}{\mathbb{R}}
\newcommand{\N}{\mathbb{N}}

\newcommand{\Q}{\mathbb{Q}}

\newcommand{\e}{\varepsilon}
\newcommand{\w}{\omega}
\newcommand{\K}{\mathcal K}
\newcommand{\C}{\mathcal{C}}
\newcommand{\F}{\mathcal{F}}
\newcommand{\pr}{\mathrm{pr}}

\newtheorem{theorem}{Theorem}[section]
\newtheorem{lemma}[theorem]{Lemma}
\newtheorem{corollary}[theorem]{Corollary}
\newtheorem{proposition}[theorem]{Proposition}

\newtheorem{fact}[theorem]{Fact}

\theoremstyle{definition}

\newtheorem{examples}[theorem]{Examples}

\theoremstyle{remark}

\newtheorem{remark}[theorem]{Remark}
\newtheorem{remarks}[theorem]{Remarks}

\newtheorem{claim}{Claim}[theorem]

\numberwithin{equation}{section}

\begin{document}

\title{Hyperspaces of countable compacta}

\author{Taras Banakh}
\email{t.o.banakh@gmail.com}
\address{Ivan Franko National University of Lviv (Ukraine) and Jan Kochanowski University in Kielce (Poland)}

\author{Pawe{\l}  Krupski}

\email{pawel.krupski@pwr.edu.pl}
\address{Faculty of Pure and Applied Mathematics, Wroc{\l}aw University of Science and Technology, Wybrze\.{z}e Wys\-pia\'n\-skiego 27, 50-370 Wroc{\l}aw, Poland.}

\author{Krzysztof Omiljanowski}
\email{Krzysztof.Omiljanowski@math.uni.wroc.pl}
\address{Mathematical Institute, University of Wroc{\l}aw, Pl. Grunwaldzki 2/4, 50-384 Wroc{\l}aw, Poland}
\date{\today}

\begin{abstract}
 Hyperspaces $\mathcal H(X)$ of all countable compact subsets of a metric space $X$ and $\mathcal A_n(X)$    of infinite  compact subsets   which have at most $n$ ($n\in\N$), or finitely many ($n=\w$) or countably many ($n=\w+1$) accumulation points  are studied. By  descriptive set-theoretical methods, we  fully characterize them  for 0-dimensional, dense-in-itself, Polish spaces and partially for $\sigma$-compact spaces $X$. Using the theory of absorbing sets, we get characterizations of $\mathcal H(X)$, $\mathcal A_\w(X)$ and $\mathcal A_{\w+1}(X)$ for  nondegenerate connected, locally connected Polish spaces $X$ which are either locally compact or nowhere locally compact. For every  $n\in\N$, we show that if  $X$ is an interval or a simple closed curve, $\mathcal A_n(X)$ is homeomorphic to the linear space $c_0=\{(x_i) \in\mathbb R^{\w}: \lim x_i=0\}$ with the product topology;
 if $X$ is a Peano continuum and a point $p\in X$ is of order $\ge 2$,  then the hyperspace $\mathcal A_1(X,\{p\})$ of all compacta with exactly one accumulation point $p$ also is  homeomorphic to $c_0$.

\end{abstract}
\maketitle

\section{Introduction}
All spaces in the paper are metric.

Let  $\mathcal K(X)$ be the hyperspace of all nonempty compact subsets of $X$ with the Vietoris topology. It is well known that $\mathcal K(X)$ shares many basic topological properties of space $X$ like, e.g.,  completeness, local compactness, compactness, connectedness, local  connectedness, dimension 0. Recall also that $\mathcal K(X)$ is an absolute neighborhhood retract (ANR) if and only if $X$ is locally continuum-connected and it is an absolute retract (AR) if, additionally,  $X$ is connected~\cite{Cu0}; if $X$ is nondegenerate noncompact, locally compact, locally connected (connected) then $\mathcal K(X)$ is an $\I^\w$-manifold ($\cong\I^\w\setminus\{point\}$)~\cite{Cu0}. For a nondegenerate Peano continuum $X$, $\K(X)  \cong  \I^\w$ (the symbol $\cong$ stands for ``homeomorphic to'').

The  hyperspace $ \mathcal F(X)\subset \K(X)$ of all finite subsets of $X$ was also extensively studied for various spaces $X$. Clearly, for the rationals $\Q$, $\mathcal F(\Q)\cong\Q$. It follows from Lemma~\ref{l:Engelen} (\cite[Lemma 3.1]{E2}) that $\mathcal F(\R\setminus\Q)\cong\Q\times(\R\setminus\Q)$ and $\mathcal F(\{0,1\}^\w)\cong\Q\times \{0,1\}^\w$. If $X$ is  locally path-connected (and connected) then   $\mathcal F(X)$ is  an ANR (AR) which is homotopy dense in $\mathcal K(X)$~\cite{CuN}; for a nondegenerate Peano continuum  $X$, $\mathcal F(X)\cong [0,1]^\omega\setminus (0,1)^\omega$~\cite{Cu}.

Another interesting subspace of $\mathcal K(X)$  is the hyperspace $\mathcal H(X)$ of all nonempty at most countable compacta which seems to have been  less recognized. In general, if $X$ is an uncountable Polish space, then $\mathcal H(X)$ is $\mathbf \Pi^1_1$-complete~\cite[Theorem (27.5)]{Ke} (in such case we will call it the \emph{Hurewicz set for} $X$).  The hyperspace $\mathcal H(\Q)$ was characterized by H. Michalewski~\cite{Mi} as a first category, zero-dimensional, separable, metrizable space with the property that every nonempty clopen subset is $\mathbf \Pi^1_1$-complete.
For $X=\I$, $\mathcal H(X)$  was fully topologically characterized by R. Cauty as a $\mathbf \Pi^1_1$-absorbing set, so it  is homeomorphic, e.g., to the space of all differentiable functions $\I\to\R$ with the uniform convergence~\cite{C1}.

Denote by  $A'$  the derived set of all accumulation points of $A\in \mathcal K(X)$. In this paper we mainly study   the following natural hyperspaces of countable  compacta in $X$:
\begin{itemize}
\item
$\mathcal A_n(X)=\{A\in \mathcal K(X): 1\le |A'|\le n\}$ ($n\in\N$),
\item
 $\mathcal A_\omega(X)= \{A\in \mathcal K(X): 1\le |A'|< \w\}$.
\item
$\mathcal H(X)$ and $\mathcal A_{\w+1}(X)=\{A\in \mathcal K(X): 1\le |A'|\le \w\}=\mathcal H(X)\setminus \F(X)$.

\end{itemize}
Our first task is to evaluate the descriptive complexity of the hyperspaces.
Clearly, $\mathcal A_{\w+1}(X)$ is a $G_\delta$-subset of  $\mathcal H(X)$, so  whenever   $\mathcal H(X)$ is absolute coanalytic, $\mathcal A_{\w+1}(X)$ is also such.
 Evaluations of absolute or exact Borel classes of $\mathcal A_n(X)$ is more complicated. For any Polish space $X$ without isolated points, $\mathcal A_n(X)$ is true absolute $F_{\sigma\delta}$ and $\mathcal A_\omega(X)$ is true absolute $F_{\sigma\delta\sigma}$. For $X=\Q$,  $\mathcal A_n(\Q)$ is in the small Borel class $D_{2n}(F_{\sigma\delta})$ in $\mathcal K(\R)$ but we do not know if it is absolute $F_{\sigma\delta}$.

 Next, we characterize $\mathcal A_n(X)$, $n\in \N$,  for any 0-dimensional Polish space which is dense-in-itself (i.e., without isolated points) as the infinite  product $\Q^\w$. If $X$ is a 0-dimensional $\sigma$-compact metric space without isolated points  then
 $$\mathcal A_n(X,F):=\{A\in \A_n(X): A'\subset F\}\cong\mathbb Q^\w$$
  for any $F\in\K(X)$ of cardinality $|F|\ge n$. In particular, the hyperspace $\mathcal A_1(\Q,\{q\})=\{A\in \mathcal A_1(\Q): A'=\{q\}\}$ also is homeomorphic to  $\Q^\w$. Thus, we get a partial answer to  the question asked in~\cite{GO} if  $\mathcal A_1(\mathbb R \setminus  \mathbb Q)$ is homeomorphic to $\mathcal A_1(\mathbb Q)$. The full positive answer is equivalent to the $F_{\sigma\delta}$-absoluteness of $\mathcal A_1(\Q)$ which remains an open problem. 

We show that for  a dense-in-itself, 0-dimensional space $X$, the hyperspace $\mathcal A_\omega(X)$  is homeomorphic to the standard, everywhere $\mathbf \Pi^0_4$-complete  set $S_4\subset \{0,1\}^w$  in two cases: $X$ a Polish space or a $\sigma$-compact metric space.

If $X$ is a dense-in-itself, 0-dimensional Polish space, then the hyperspaces $A_{\w+1}(X)$ and $\mathcal H(X)$   are homeomorphic to $\mathcal H(\Q)$.

\

 Studying hyperspaces of compacta of reasonably nice spaces of positive dimensions, we unavoidably enter into   infinite-dimensional topology. Here, we  intensively employ the theory of absorbing sets. We describe (apparently new) an  $F_{\sigma\delta}$-absorber $\Pi_3$ and an  $F_{\sigma\delta\sigma}$-absorber $ \Sigma_4$  in the Hilbert cube and   our main results in Section~\ref{s:abs} are the following characterizations:
 $$\A_\w(X)\cong \Sigma_4,\quad A_{\w+1}(X)\cong \mathcal H(X)\cong \mathcal H(\I)$$  if
 $X$ is  nondegenerate, connected, locally connected and either (1) locally compact or (2) Polish, nowhere locally compact.

 One of  the simplest examples of   type (2)-spaces is  $\R^2\setminus \Q^2$.   Other natural examples of such spaces   include the set of all ``irrational points'' of the Sierpiński carpet, infinite countable products of non-compact intervals, N\"{o}beling or Lipscomb universal spaces of dimension $\ge 1$.
In particular, the characterization  extends  the Cauty's characterization of $\mathcal H(\I)$  over all spaces $X$ as in (1) and (2).

The hyperspace $\mathcal A_n(X)$  is more difficult to handle. In Sections~\ref{interval} and~\ref{spheres}, the following characterizations are obtained:
$$\A_n(\I)\cong \A_n((0,1))\cong \A_n(S^1)\cong\Pi_3\cong c_0:=\{(x_i)\in\R^\w: \lim x_i=0\}.$$
Incidentally, the  characterizations  answer~\cite[Question 2.17]{GO} and  a question in~\cite{CMP} if  $\A_n(S^1)$ is contractible.

Finally, in Section~\ref{Peano}, we show that if $X$ is a Peano continuum with a point $p$ of order $\ge 2$ and $p\in F\in\F(X)$, then
 $$\{A\in \mathcal A_n(X): F\subset A\}\cong \mathcal A_1(X,\{p\})\cong c_0.$$

\section{Borel complexity of  $\mathcal A_n(X)$ and $\mathcal A_\omega(X)$}

Let us recall the standard notations of absolute Borel classes of spaces:
\begin{itemize}
\item for a countable ordinal $\alpha\ge 1$,  $\mathbf \Pi^0_\alpha$ is the absolute $\alpha$-th multiplicative class (i.e., $\mathbf\Pi^0_1$ is the class of compact metrizable spaces, $\mathbf \Pi^0_2$ is the class of Polish spaces, etc.);
\item for $\alpha\ge 2$, $\mathbf\Sigma^0_\alpha$ is the absolute   $\alpha$-th additive class (i.e., $\mathbf\Sigma^0_2$ is the class of $\sigma$-compact spaces, $\mathbf \Sigma^0_3$ is the class of absolute $G_{\delta\sigma}$-spaces, etc.).
\end{itemize}
The class of absolute coanalytic spaces is denoted by $\mathbf \Pi^1_1$.

 Let $\Gamma$ be a family of subsets of  $X$. For a natural number $n$, let $D_{2n}(\Gamma)$
 be the family of sets of type $\bigcup_{k=1}^n(A_{2k}\setminus A_{2k-1})$ where $(A_k)_{k=1}^{n}$ is an increasing sequence of sets from $\Gamma$.
For the class $\Gamma$ of $F_{\sigma\delta}$-sets, elements of the small Borel class $D_{2n}(\Gamma)$ will be called sets of difference type $D_{2n}(F_{\sigma\delta})$.

If $\mathcal P$ is a topological property, then  a subspace $Y\subset Z$ is   everywhere (nowhere) $\mathcal P$ if every (no) nonempty, relatively  open subset of $Y$ has $\mathcal P$.

\

 For $E,F\subset X$ and $n\in\mathbb N$, denote:
\begin{itemize}
\item
 the closed subspace $\mathcal F_n(X)=\{A\in \mathcal K(X): |A|\le n\}$ of $\mathcal K(X)$,
  \item
 $\mathcal A_{=n}(X)=\{A\in \mathcal K(X):  |A'|= n\}$,
 \item
 $\mathcal A_n(E,F)=\{A\in \mathcal A_n(X):A\setminus A'\subset E \wedge  A'\subset F\}.$
 \item
 $\mathcal A_{=n}(E,F)=\{A\in \mathcal A_{=n}(X):A\setminus A'\subset E \wedge  A'\subset F\}.$
\item
 $\K(X)^F=\{K\in\K(X):  F\subset K\},\quad \mathcal A_n(X)^F=\A_n(X)\cap \K(X)^F.$
 \item
 $\K(X)_F=\{K\in\K(X):  F\cap K\neq\emptyset\},\quad \mathcal A_n(X)_F=\A_n(X)\cap \K(X)_F.$
 \end{itemize}

\

\

\begin{theorem}\label{p'}
For each $n\in\N$,
\begin{enumerate}
\item
$\mathcal A_n(X)$ and $\mathcal A_n(X,F)$, for any $F\in\K(X)$, are  $F_{\sigma\delta}$-sets in $\mathcal K(X)$;
\item
 $\mathcal A_\omega(X)$ is $F_{\sigma\delta\sigma}$ in $\mathcal K(X)$;
\newline
if $X$ is metric separable and $F$ is an $F_\sigma$-subset of $X$, then
\item
  $\A_{=n}(F,X)$ is of type $F_{\sigma\delta}$ in $\A_{=n}(X)$;

  \item
  $\A_{=n}(X,F)$ is of type $G_{\delta\sigma}$ in $\A_{=n}(X)$;
 \item
  $\A_{=n}(F)$ is of type $D_2(F_{\sigma\delta})$ in $\A_{=n}(X)$;
 \item
$\A_{n}(F,X)$,  $\A_{n}(X,F)$ and $\A_{n}(F)$ are of type $D_{2n}(F_{\sigma\delta})$ in $\mathcal K(X)$.
\end{enumerate}
\end{theorem}

\begin{proof}

K. Kuratowski proved in~\cite{Ku} that the derived set map
$$\mathbf D: \mathcal K(X)\setminus \F(X)\to \mathcal K(X),\qquad \bold D(A)= A'$$
is Borel of the second class. Actually, it is convenient to consider $\bold D$  as a map from the whole $\mathcal K(X)$ to the space $\mathcal K(X)\cup \{\emptyset\}$ with the isolated point $ \{\emptyset\}$ and then a direct proof in~\cite{CM} of  Kuratowski's theorem shows that  the preimage under $\bold D$ of each closed set in $\mathcal K(X)$ is $F_{\sigma\delta}$ in $\mathcal K(X)$.

Observe that
$$\mathcal A_n(X)=\bold D^{-1}(\F_n(X))\quad\text{ and}\quad \mathcal A_n(X,F)=\bold D^{-1}(F)\cap  \mathcal A_n(X)$$ which establishes (1)  and yields (2).

\

Now,
fix a metric $d$ generating the topology of $X$. For $x\in X$ and $\e>0$ denote by
$$B(x,\e)=\{y\in X:d(y,x)< \e\}\quad\mbox{and}\quad B[x,\e]=\{y\in X:d(y,x)\le \e\}$$  the open and closed $\e$-balls centered at $x$. For  $A\subseteq X$, let $$B(A,\e)=\bigcup_{a\in A}B(a,\e)\quad\mbox{and}\quad B[A,\e]=\bigcup_{a\in A}B[a,\e].$$ Fix a countable dense set $D$ in $X$ and $n\in\N$.

\

(3). The equality
$$
\A_{=n}(F,X)=\bigcap_{m\in\N}\bigcup_{A\in[D]^{\le n}}\bigcup_{k\in\N}\{K\in \A_{=n}(X):K\setminus B[A,\tfrac1m]\in [F_k]^{\le k}\}$$
witnesses that the sets $\A_{=n}(F,X)$ is of type $F_{\sigma\delta}$ in $\A_{=n}(X)$.

\

(4). The equality
\begin{multline*}
\A_{=n}(X)\setminus \A_{=n}(X,F)=\\
\bigcap_{k\in\N}\bigcup_{A\in [D]^{<n}}\bigcup_{m\in\N}\{K\in\A_{=n}(X):K\cap B(F_k,\tfrac1m)\setminus B[A,\tfrac1k]\in [X]^{\le m}\}
\end{multline*}
shows that the set $\A_{=n}(X)\setminus \A_{=n}(X,F)$
is of type $F_{\sigma\delta}$ in $\A_{=n}(X)$ and hence $\A_{=n}(X,F)$ is of type $G_{\delta\sigma}$ in $\A_{=n}(X)$.

\

(5). Since $\A_{=n}(F)=\A_{=n}(F,X)\cap\A_{=n}(X,F)$, the set $\A_{=n}(F)$ is of type $D_2(F_{\sigma\delta})$ in $\A_{=n}(X)$ according to the preceding two statements.

\

(6). For every $k\in\N$, choose an $F_{\sigma\delta}$-set $S_{2k}$ in $\K(X)$ such that
$$S_{2k}\cap \A_{=k}(X)=\A_{=k}(F,X).$$
Since $\A_k(X)$ and $\A_{k-1}(X)$ are $F_{\sigma\delta}$-sets in $\K(X)$, we can assume that $\A_{k-1}(X)\subseteq S_{2k}\subseteq \A_k(X)$. Put $S_1=\A_0(X)\setminus\A_0(F)$ and $S_{2k-1}=\A_{k-1}(X)$ for $k>1$. Since $\A_n(F,X)=\bigcup_{k=1}^n(S_{2k}\setminus S_{2k-1})$, the set $\A_n(F,X)$ is of type $D_{2n}(F_{\sigma\delta})$ in $\K(X)$.

For every $k>1$ put $T_{2k}=\A_{k}(X)$ and choose an $F_{\sigma\delta}$-set $T_{2k-1}$ in $\K(X)$ such that
$$T_{2k-1}\cap \A_{=k}(X)=\A_{=k}(X)\setminus \A_{=k}(X,F).$$
 Since $\A_k(X)$ and $\A_{k-1}(X)$ are $F_{\sigma\delta}$-sets in $\K(X)$, we can assume that $\A_{k-1}(X)\subseteq T_{2k-1}\subseteq \A_k(X)$.  Put also
 $$T_2=\A_1(X),\quad T_1=\A_1(X)\setminus \A_1(X,F)=\A_{=1}(X)\setminus\A_{=1}(X,F).$$
 Since
$$\A_n(X,F)=\bigcup_{k=1}^{n}(T_{2k}\setminus T_{2k-1}),$$ the set $\A_n(X,F)$ is of type $D_{2n}(F_{\sigma\delta})$ in $\K(X)$.

\

 Since $\A_{=1}(X)=\A_1(X)$ is of type $F_{\sigma\delta}$ in $\K(X)$,  the set $\A_{=1}(F,X)\setminus\A_{=1}(X,F)$ is of type $F_{\sigma\delta}$ in $\A_{=1}(X)$ (by  statements (3) and (4)) and hence in $\K(X)$. It means that the set $$Q_1=\A_{1}(F,X)\setminus \A_1(F)=\A_{=1}(F,X)\setminus\A_{=1}(X,F)$$ is of type $F_{\sigma\delta}$ in $\K(X)$.
  For $k>1$, let $Q_{2k-1}=S_{2k}\cap T_{2k-1}$.

 Since
$$\A_n(F)=\bigcup_{k=1}^n(S_{2k}\setminus Q_{2k-1}),$$ the set $\A_n(F)$ is of type $D_{2n}(F_{\sigma\delta})$ in $\K(X)$.
\end{proof}

\begin{corollary}\label{Polish}
If $X$ is a  Polish space then the Borel classes of sets in (1), (2), (6) of Theorem~\ref{p'} are absolute.
\end{corollary}

In order to evaluate Borel classes of $\mathcal A_n(X)$ and $\mathcal A_\omega(X)$ from below, we will exploit the standard sets $P_3$ and $S_4$ which are $\mathbf \Pi^0_3$-complete and $\mathbf \Sigma^0_4$-complete, respectively~\cite[Exercise (23.1), Exercise (23.6)]{Ke}. They can be represented in slightly different from~\cite{Ke} but equivalent forms:
\begin{itemize}
\item
$P_3=\{(x_i)_{i\in\N}\in \{0,1\}^\N:\forall j\in\w\;\forall^\infty k\in\w\;\;(x_{2^j(2k+1)}=0)\}$
\item
$S_4=\{(x_i)_{i\in\N}\in \{0,1\}^\N:\forall ^\infty j\in\w\forall ^\infty k\in\w \, \;\;(x_{2^j(2k+1)}=0)\}$.
\end{itemize}

\

\begin{theorem}\label{p1}
  If $X$ contains a compactum  of the Cantor-Bendixson rank $3$, then  $\mathcal A_n(X)$ is not $G_{\delta\sigma}$ in $\mathcal K(X)$ and $\mathcal A_\omega(X)$ is not $G_{\delta\sigma\delta}$ in $\mathcal K(X)$. If  $X$ is dense-in-itself then
 \begin{enumerate}
 \item
 $\mathcal A_n(X)$ and $\mathcal A_\omega(X)$ are nowhere  $G_{\delta\sigma}$ and nowhere $G_{\delta\sigma\delta}$ in $\mathcal K(X)$, respectively,

\item  If $F\in \mathcal K(X)$ and $|F|\ge n$ then $\mathcal A_n(X,F)$ is  nowhere  $G_{\delta\sigma}$ in $\K(X)_F$
\item
 If $F\in \mathcal K(X)$  is infinite  then  $\mathcal A_{\w}(X,F)$ is  nowhere  $G_{\delta\sigma\delta}$ in $\K(X)_F$.
\end{enumerate}
\end{theorem}

\begin{proof}

Let the Cantor-Bendixson rank of some $A\in \mathcal K(X)$ equals $3$. Without loss of generality we can assume that
$A= \cl\bigl\{ 2^{-j}+{2^{-(j+k)}}: j,k\in \w \bigr\}.$
For any $n\in\N$,  put
\begin{equation}\label{e:chi}
\chi(n)=\left\{
                \begin{array}{ll}
                  \cl\bigl\{2^{-j}+2 ^{-(j+k)}: 1\le j\le n-1, k\in\w\bigr\}, & \hbox{\text{if $n>1$};} \\
                  \emptyset, & \hbox{\text{if $n=1$}.}
                \end{array}
              \right.
\end{equation}
Define a continuous map $\psi_n:  \{0,1\}^{\N}\to \mathcal K(X)$ by
\begin{equation*}
    \psi_n(x)=  \chi(n) \cup
\cl\bigl(\bigl\{2^{-j}+2^{-(j+k)}x_{2^{j-n}(2k+1)}:j\ge n, k\in\w\bigr\}\bigr)\subset A.
    \end{equation*}
 One  easily checks that $$\psi_n^{-1}(\mathcal A_n(X))=P_3\quad\text{and}\quad \psi_1^{-1}(\mathcal A_\omega(X))=S_4.$$
  This guarantees that $\mathcal A_n(X)$ is not $G_{\delta\sigma}$ and $\mathcal A_\omega(X)$ is not $G_{\delta\sigma\delta}$.

\

Now, assume that $X$ has no isolated points. Then every nonempty open subset of $X$  contains a copy of $A$. Consider a basic open set $\mathcal U$ in the Vietoris topology in $\mathcal K(X)$:
 $$\mathcal U= \langle U_1,\dots,U_k\rangle =\{K\in \mathcal K(X): K\subset \bigcup_{i=1}^k U_i,\, (\forall i) \, U_i\cap K\neq\emptyset\}, $$
where $U_i$'s are open  subsets of $X$ and pick points $u_i\in U_i$ for $i=1,\dots, k$.

(1).
Assume, without loss of generality, that $A\subset U_1$. Then
$$A\cup \{u_1,\dots,u_k\}\in \mathcal U $$
and the map $$\widetilde{\psi_n}:  \{0,1\}^{\N}\to \mathcal U,\quad \widetilde{\psi_n}(x)=\psi_n(x)\cup \{u_1,\dots,u_k\}$$  satisfies
$$(\widetilde{\psi_n})^{-1}(\mathcal A_n(X)\cap \mathcal U)=P_3\quad\text{and}\quad (\widetilde{\psi_1})^{-1}(\mathcal A_\omega(X)\cap \mathcal U)=S_4$$
 which completes the proof of (1).

(2) and (3).  Let $\mathcal U_F= \langle U_1,\dots,U_k\rangle \cap \K(X)_F$.  We can assume that  $A\subset \bigcup_{i=1}^k U_i$ and  $\{0\}\cup\{2^{-j}:j=1,\dots,n-1\}\subset F$  in case (2) and $\{0\}\cup\{2^{-j}:j\in\N\}\subset F$ in case (3).  Then
$$(\widetilde{\psi_n})^{-1}(\mathcal A_n(X,F)\cap \mathcal U_F)=P_3\quad\text{and}\quad (\widetilde{\psi_1})^{-1}(\mathcal A_\w(X,F)\cap \mathcal U_F)=S_4$$ in respective cases.

\end{proof}

The following general fact can also be observed.

\begin{proposition}\label{anal}
Let $\mathbf\Gamma$ be any absolute Borel class containing class $\mathbf\Sigma^0_2$  or any projective class.
If an uncountable  metrizable separable space $X$ is   not in $\mathbf \Gamma$, then $\mathcal A_n(X)$ is not in $\mathbf\Gamma$ for each $n\in\mathbb N\cup\{\omega\}$.
\end{proposition}

\begin{proof}
Suppose $\A_n(X)$ belongs to $\mathbf\Gamma$. Since $X$ is uncountable, it contains infinitely many accumulation points. So, we can find $K\in \A_n(X)$.  Consider the continuous map
$$\delta: X\to \A_n(X),\quad \delta(x)=\{x\}\cup K$$
and observe that the image $\delta(X)$ is closed in $\A_n(X)$, hence it is in $\mathbf\Gamma$. Then    $\delta(X)\setminus \{K\}$ belongs also to $\mathbf\Gamma$.
Since $\delta\upharpoonright_{ X\setminus K} : X\setminus K \to \delta(X)\setminus \{K\}$ is a homeomorphism, the space $X\setminus K$ is in $\mathbf\Gamma$ and so is the space $X=(X\setminus K)\cup K$, a contradiction.

\end{proof}

\section{Hyperspaces $\mathcal A_n(X)$ for 0-dimensional $X$}

In this section, we  characterize hyperspaces $\mathcal A_n(X),$ $n\le \w$, for 0-dimensional Polish or $\sigma$-compact  spaces without isolated points.

\begin{lemma}\label{l:P}
$P_3\cong\Q^\w$. In particular, $P_3$ is of the first category (in itself) and nowhere $G_{\delta\sigma}$.
\end{lemma}

\begin{proof}
Represent $\N$ as the countable disjoint union $\N=\bigcup_{n\in \w} N_n,$ where $N_n=\{2^n(2k+1): k\in\w\}$ and let $\pr_{N_n}: \{0,1\}^\N\to \{0,1\}^{N_n}$ be the projection. Since the set $\{x\in\{0,1\}^\N: \forall^\infty i \,(x_i=0)\}$ is homeomorphic to $\Q$, the sets $Z_n=\pr_{N_n}(P_3)$ are also homeomorphic to $\Q$. For each $n\in\w$, fix a homeomorphism $h_n:Z_n\to \Q$ and let a homeomorphism
$h:P_3\to  \Q^\w$ be defined by
$$h(x)=\bigl(h_n(\pr_{N_n}(x))\bigr)_{n\in\w}.$$

\end{proof}

\begin{lemma}\label{l:S4}
 The set $S_4$ is of the first category (in itself) and strongly homogeneous (i.e., every two nonempty clopen subsets are homeomorphic). In particular, $S_4$ is nowhere $G_{\delta\sigma\delta}$.
 \end{lemma}
 \begin{proof}
 For each $m\in\w$, put $$T_m=\{x\in\{0,1\}^\N: (\exists\,n\ge m)\,(\forall\, r\ge n)\, (\forall^\infty\,k)\,(x_{2^r(2k+1)}=0)\}.$$
 Observe that
 \begin{itemize}
 \item
 $S_4=\bigcup_m T_m$,
 \item
 $T_m\subset T_{m+1}$,
 \item
 $T_m$ is closed in $S_4$,
 \item
 $T_m$ is nowhere dense in $T_{m+1}$.
 \end{itemize}
 It follows that $S_4$ is of the first category.

 The strong homogeneity follows directly from the definition of $S_4$.
 \end{proof}

\

The following  theorem is due to Steel~\cite[Theorem 2]{S} and van Engelen~\cite[Theorem 4.6]{E1}.

\begin{lemma}\label{l:S}
If $\alpha\ge 3$, then any two 0-dimensional, metric separable, first category  spaces from the class $\mathbf\Pi^0_\alpha$ ($\mathbf\Sigma^0_\alpha$) which are nowhere
$\mathbf\Sigma^0_\alpha$ ($\mathbf\Pi^0_\alpha$, resp.) are homeomorphic.
\end{lemma}

It was shown in~\cite{GO} that $\mathcal A_1(X)$ is of the first category for any second countable topological space $X$. We provide a quick argument for this fact valid for any $\mathcal A_n(X)$ in the case of a dense-in-itself,  metric, separable, 0-dimensional space $X$.

\begin{lemma}\label{le}
Let $X$ be a dense-in-itself, metric, separable, 0-dimensional space, $n\le\omega$ and $F\in \mathcal K(X)$ be a set of cardinality $\ge n$.  Then the hyperspaces  $\mathcal A_n(X)$, $\mathcal A_n(X,F)$, $\mathcal A_{\w+1}(X)$ and $\mathcal H(X)$ are of the first category.
\end{lemma}
\begin{proof}
Let $\A$ denote any of the hyperspaces. Let $\{B_1,B_2,\ldots\}$ be a clopen base in $X$ which is closed under finite unions and $\mathcal B_{i,k}$ be the family of all $A$ in $\mathcal A$  such that $ 1\le |A\setminus B_i|\le k\}.$
Since $X$ has no isolated points, the sets $\mathcal B_{i,k}$ are nowhere dense in  $\mathcal A$. Clearly, $\mathcal A$  is the union $\bigcup_{i,k} \mathcal B_{i,k}$.
\end{proof}

\

Now, we get the following characterizations.

\begin{theorem}\label{tPolish}
If $X$ is a dense-in-itself, 0-dimensional Polish space then
\begin{enumerate}
\item
$\mathcal A_n(X)\cong\mathcal A_n(X,F)\cong P_3\cong\mathbb Q^\w$ for each $n\in\N$ and $F\in\K(X)$ of cardinality $\ge n$;
\item
$\A_\omega(X)\cong \A_\w(X,F)\cong S_4$ for each infinite $F\in\K(X)$;
\item
$\A_{\w+1}(X)\cong \mathcal H(X)\cong \mathcal H(\Q)$.
\end{enumerate}
\end{theorem}

\begin{proof}
Parts (1) and (2) follow from Corollary~\ref{Polish}, Theorem~\ref{p1} and Lemmas~\ref{l:S4}--\ref{le}.

For part (3),  observe that  nonempty clopen subsets of $\A_{\w+1}(X)$ and $\mathcal H(X)$ are  $\mathbf \Pi^1_1$-complete. To see this,
let $\mathcal U=\langle U_1,\ldots,U_k\rangle$ be a basic clopen set in $\K(X)$, where  $U_1,\ldots,U_k$ are basic clopen subsets of $X$.
Take a countable dense subset $Q$ in $\bigcup_{i=1}^k U_i$ and choose $A_1\in\A_1(U_1)$. The set $\mathcal U$ is a Polish 0-dimensional space containing the Hurewicz set $\mathcal H(Q)\cong \mathcal H(\Q)$. Now, the  continuous map
$f:\mathcal U \to \mathcal U,$ $f(A)=A\cup A_1$
 satisfies $$f^{-1}\bigl(\A_{\w+1}(X)\cap \mathcal U\bigr)=f^{-1}\bigl(\mathcal H(X)\cap \mathcal U\bigr)=\mathcal H(Q).$$
 Thus, by Lemma~\ref{le} we can apply the Michalewski's characterization~\cite{Mi}.

\end{proof}
\

\begin{theorem}\label{Ksigma}
If $X$ is a 0-dimensional  $\sigma$-compact metric  space and $F\in \K(X)$, then $\mathcal A_n(X,F)$ is in $\mathbf \Pi^0_3$ for each $n\in\mathbb N$ and $\mathcal A_\omega(X,F)$ is in $\mathbf \Sigma^0_4$.
\end{theorem}

\begin{proof}
We can assume that $X$ is contained in the Cantor set $C$.
Consider the derived set operator $\bold D$ on $\mathcal K(C)$.
The preimage $D^{-1}\left(\mathcal F_n\right)$ is $F_{\sigma\delta}$ in $\mathcal K(C)$ for each  $n\in\mathbb N$ and $F_{\sigma\delta\sigma}$ for $n=\omega$.
Let $C=U_1\supsetneqq
U_2\supsetneqq\dots $ be clopen subsets of $C$ whose intersection is $F$. The sets $X_j=X\cap(U_{j}\setminus U_{j+1})$ are $\sigma$-compact. It follows that sets $\mathcal F(X_j)$ are also $\sigma$-compact. Hence, each $\mathcal F(X_j)\cup \{\emptyset\}$ as a subset of a compact space $\mathcal K(C)\cup\{\emptyset\}$ is  $\sigma$-compact.

Let
$$\mathcal B_j=\{A\in \mathcal K(C): A\cap (U_j\setminus U_{j+1})\in \mathcal F(X_j)\cup \{\emptyset\}\}.$$

The intersection map
$$\Phi: \mathcal K(C)\to \mathcal K(C)\cup \{\emptyset\}, \quad \Phi(A)=A\cap (U_j\setminus U_{j+1})$$ is continuous~\cite[Theorems 2 and 3, p.180]{KuI},
consequently,
$$\mathcal B_j=\Phi^{-1}\left(\mathcal F(X_j)\cup \{\emptyset\}\}\right)$$ is $F_\sigma$ in     $\mathcal K(C)$.

Observe that

$$\mathcal A_n(X,F) = D^{-1}\left(\mathcal F_n(K_i)\right)\cap \mathcal A_n(X)= D^{-1}\left(\mathcal F_n(K_i)\right)\cap \bigcap_{j=1}^\infty \mathcal B_j $$
which shows that
\begin{equation*}\label{eq:Ki}
\text{$\mathcal A_n(X,F)$ is  $F_{\sigma\delta}$ in $\mathcal K(C)$ for $n\in\mathbb N$,  and is  $F_{\sigma\delta\sigma}$ in $\mathcal K(C)$ for $n=\omega$}.
\end{equation*}

\end{proof}

\

\begin{theorem}\label{Ksigma1}
\begin{enumerate}
\item
If $X$ is a dense-in-itself, 0-dimensional $\sigma$-compact metric space  then  $\mathcal A_n(X,F)\cong\mathbb Q^\w$ for any $F\in\K(X)$ of cardinality $|F|\ge n$;

\item
$\A_w(X)\cong\A_w(X,F)\cong S_4$ for  any infinite $F\in\mathcal K(X)$.
\end{enumerate}
\end{theorem}

\begin{proof}
(1).  $\mathcal A_n(X,F)$ is in $\mathbf \Pi^0_3$ by Theorem~\ref{Ksigma}. Theorem~\ref{p1} and Lemma~\ref{le} guarantee that  the hyperspace is nowhere $\mathbf \Sigma^0_3$  and of the first category, so Lemma~\ref{l:S} applies.

(2). The space $X$ can be considered as an $F_\sigma$-subset of the Cantor set $C$. By Theorem~\ref{p'} (6), the hyperspaces $\A_n(X)$, $n\in\N$, are $F_{\sigma\delta\sigma}$-subsets of $\K(C)$, so  $\A_w(X)$ also is $F_{\sigma\delta\sigma}$ in $\K(C)$, hence $\A_w(X)$ is in $\mathbf\Sigma^0_4$.
 $\A_w(X,F)$ is in $\mathbf \Sigma^0_4$ by Theorem~\ref{Ksigma}. Both  hyperspaces are nowhere $\mathbf \Pi^0_4$ (Theorem~\ref{p1}) and of the first category (Lemma~\ref{le}), hence they are homeomorphic to $S_4$ by Lemma~\ref{l:S}.
 \end{proof}

\begin{corollary}\label{cor:Q}\begin{enumerate}
\item
$\mathcal A_n(\Q,F)\cong\A_n(\R\setminus \Q)\cong \mathcal A_n(\R\setminus \Q,F')\cong \Q^\w$ for any $F\in\K(\Q)$ and $F'\in \K(\R\setminus \Q)$  of cardinalities $\ge n$.
\item
$\A_w(\Q)\cong\A_w(\Q,F)\cong \A_w(\R\setminus \Q)\cong\mathcal A_w(\R\setminus \Q,F')\cong S_4$ for  any infinite $F\in\mathcal K(\Q)$ and $F'\in \K(\R\setminus \Q).$
\end{enumerate}
\end{corollary}

\

In view of Lemma~\ref{l:S} and by  Theorem~\ref{tPolish},  Theorem~\ref{p1} and Lemma~\ref{le}, the  question asked in~\cite{GO} if $\mathcal A_1(\mathbb R \setminus  \mathbb Q)$ is homeomorphic to $\mathcal A_1(\mathbb Q)$ reduces to the problem of the $F_{\sigma\delta}$-absolutness of $\mathcal A_1(\mathbb Q)$ (more generally, of $\A_n(\Q)$); equivalently,   whether or not $\mathcal A_1(\mathbb Q)$ is   $F_{\sigma\delta}$ in $\mathcal K(C)$, where $C$ is a Cantor set.
Aiming at this direction, we  observe several   facts shedding some light on the structure of $\mathcal A_n(\mathbb Q)$.

We use the following lemma due to  van Engelen~\cite[Lemma 3.1]{E2}.

\begin{lemma}\label{l:Engelen}
Let  $X$ and $Y$ be  0-dimensional metric separable spaces, $X=\bigcup_{i=1}^\infty X_i$, $Y=\bigcup_{i=1}^\infty Y_i$  with  $X_i$ (resp. $Y_i$) closed and nowhere dense in $X$ (resp. $Y$) and let every nonempty clopen subset of $X$ (resp. $Y$) contain a closed nowhere dense copy of each $Y_i$ (resp $X_i$). Then $X\cong Y$.
\end{lemma}

\

\begin{lemma}\label{l:copy}
Every nonempty clopen subset of   $\mathcal A_n(\Q)$ contains a closed copy of $\mathcal A_n(\mathbb Q)$ for $n\in\mathbb N\cup\{\omega\}$.
\end{lemma}

\begin{proof}
We can assume that a clopen subset $\mathcal U$ of $\mathcal A_n(\mathbb Q)$ is of the form $\mathcal U= \langle U_1,\dots U_m\rangle$ for nonempty disjoint clopen subsets $U_i$ of $\mathbb Q$. The set
$
\mathcal A_n(U_1)
$
is a closed copy of $\mathcal A_n(\mathbb Q)$.

 Fix points $u_i\in U_i$, $i=2,\dots,m$. Then
 \begin{equation}\label{eq:copy} \{A\cup\{u_2,\dots,u_m\}: A\in \mathcal A_n(U_1)\}\quad
 \text{is a closed copy  of $\mathcal A_n(\mathbb Q)$ in
$\mathcal U$.}
\end{equation}

\end{proof}

By~\cite[Theorem 4.1]{E2} and Lemmas~\ref{le},~\ref{l:copy} we get
\begin{proposition}\label{p:prod}
$\mathcal A_n(\mathbb Q)$ is strongly homogeneous  and
$\mathcal A_n(\mathbb Q)\cong \mathcal A_n(\mathbb Q)\times \mathbb Q$ for $n\in\mathbb N\cup\{\omega\}$.
\end{proposition}

One can easily see
\begin{lemma}\label{lem}
$\mathcal A_1(\mathbb Q)^{\{q\}}=\cl\left(\mathcal A_1(\mathbb Q,\{q\}\right)$ (the closure in $\mathcal A_1(\mathbb Q)).$
\end{lemma}

\begin{proposition}\label{p:q}
$\mathcal A_1(\mathbb Q)^{\{q\}}\cong\mathcal A_1(\mathbb Q)\cong \mathcal A_1(\mathbb Q)^{\{q\}}\setminus \mathcal A_1(\mathbb Q,\{q\})$ for every $q\in\mathbb Q$.
\end{proposition}
\begin{proof}
We have:
$$\mathcal A_1(\mathbb Q)^{\{q\}}=\bigcup_{p\in\mathbb Q\setminus\{q\}} \mathcal A_1(\mathbb Q)^{\{p,q\}}\quad\text{and}\quad \mathcal A_1(\mathbb Q)= \bigcup_{p\in\mathbb Q}\mathcal A_1(\mathbb Q)^{\{p\}}.$$
Each $\mathcal A_1(\mathbb Q)^{\{p,q\}}$ is closed and nowhere dense in  $\mathcal A_1(\mathbb Q)^{\{q\}}$. Similarly, each $\mathcal A_1(\mathbb Q)^{\{p\}}$ is closed and nowhere dense in  $\mathcal A_1(\mathbb Q)$.
If $\mathcal U= \langle U_1,\dots U_m\rangle\cap \mathcal A_1(\mathbb Q)^{\{q\}}$ is nonempty for nonempty disjoint clopen subsets $U_i$ of $\mathbb Q$ and $q\in U_1$, then, as in~\eqref{eq:copy},
$$\{A\cup\{u_2,\dots,u_m\}: A\in \mathcal A_1(U_1) \cap  \mathcal A_1(\mathbb Q)^{\{q\}}\}$$
 is a closed and nowhere dense copy  of $\mathcal A_1(\mathbb Q)^{\{p\}}$ in
$\mathcal U$.
 Analogously, each nonempty clopen set $\langle U_1,\dots U_m\rangle\cap\mathcal A_1(\mathbb Q)$ in $\mathcal A_1(\mathbb Q)$ contains a closed nowhere dense copy of
$\mathcal A_1(\mathbb Q)^{\{p,q\}}$. Now apply Lemma~~\ref{l:Engelen}.

To prove the second equivalence,  represent $\mathcal A_1(\mathbb Q)^{\{q\}}\setminus \mathcal A_1(\mathbb Q,\{q\})$  as a union $\bigcup_{k\in\mathbb N} \mathcal B_k$, where
$$\mathcal B_k= \bigl\{A\in \mathcal A_1(\mathbb Q)^{\{q\}}: A\subset \{q\}\cup \bigl(\mathbb Q\setminus \bigl(q-\frac{\sqrt{2}}{k}, q+\frac{\sqrt{2}}{k}\bigr)\bigr)\bigr\}.
$$
One can easily check that each $\mathcal B_k$ is closed, nowhere dense in $\mathcal A_1(\mathbb Q)^{\{q\}}\setminus \mathcal A_1(\mathbb Q,\{q\})$ as well as it can be embedded as a closed nowhere dense subset in each nonempty clopen subset of $\mathcal A_1(\mathbb Q)$. Conversely, each nonempty clopen subset of $\mathcal A_1(\mathbb Q)^{\{q\}}\setminus \mathcal A_1(\mathbb Q,\{q\})$ contains a closed nowhere dense copy of $\mathcal A_1(\mathbb Q)^{\{p\}}$.

\end{proof}

  A map  $\mathcal A_1(\mathbb Q,\{0\})\times \mathbb Q \to \mathcal A_1(\mathbb Q)$  given by the translation $(A,q)\mapsto A+q$ is a continuous bijection (it is not a homeomorphism, though).  Hence, by Theorem~\ref{Ksigma1}, Lemma~\ref{lem} and Proposition~\ref{p:q}, we get

\begin{corollary}
$\mathcal A_1(\mathbb Q)$ is a one-to-one continuous image of $\mathbb Q^\w$. Equivalently,  $\cl\left(\mathcal A_1(\mathbb Q,\{q\}\right)$ is a one-to-one continuous image of $\mathcal A_1(\mathbb Q,\{q\})$.
 \end{corollary}

\section{Preliminaries related to strongly universal and absorbing sets}
 From now on, all spaces are assumed to be metric separable and all maps  continuous.

We  recall a basic terminology and facts related to absorbing sets. The reader is referred to~\cite{BGM, BC,  BRZ, M}  for more  details.

The standard Hilbert cube $\I^\w$ ($\I=[0,1]$) is considered  with  the metric $$d(x,y)=\sum_{k\in\w} \frac{|x_k-y_k|}{2^{k}}.$$
A map  $f:X\to Y$ is  \emph{approximated  arbitrarily closely} by  maps  with property $\mathcal P$ if for  any open cover $\mathcal U$ of $Y$ there is a map  $g:X\to Y$ with property $\mathcal P$ such that $f$ is  $\mathcal  U$-close to $g$, i.e., for each $x\in X$ there is $U\in \mathcal U$ containing  $\{f(x),g(x)\}$.

 A closed subset $B\subset X$ is a (\emph{strong}) $Z$-{\it set} in $X$ if the identity map of $X$ can be approximated arbitrarily closely by   maps $f:X\to X$ such that  $B\cap f(X) =\emptyset$ ($B\cap \cl(f(X)) =\emptyset$).

 An embedding $f:X\to Y$ is a $Z$-embedding if $f(X)$ is a $Z$-set in $Y$. A countable union of (compact)  $Z$-sets in $X$ will be called a ($\sigma$-compact) $\sigma Z$-{\it set} in $X$.

 A subset $A\subset Y$ is \emph{homotopy dense} in $Y$ if  there is a deformation $H:Y\times [0,1]\to Y$ such that $H(Y \times (0,1])\subset A$ (a deformation \emph{through} $A$).

\

\begin{fact}~\cite[Lemma 2.6]{C1}\label{f6}
If $M$ is an ANR, a subset $X\subset M$ is homotopy dense in $M$ and $Z$  is a strong $Z$-set in $M$, then $Z\cap X$ is a strong $Z$-set in $X$.
\end{fact}

 Let
$M$ be an  absolute neighborhood retract (ANR).
 It is  known that
\begin{itemize}
\item
$B$ is a $Z$-set in $M$ if and only if $M\setminus B$ is homotopy dense in $M$ (see \cite[Corollary 3.3]{Tor}),
\item
 if $M$ is completely metrizable and $B$ is a $\sigma Z$-set in $M$,  then   $M\setminus B$ is homotopy dense in $M$ (\cite[Exercise 3, p. 31]{BRZ}),
\item
 if $A$ is homotopy dense in $M$ then $A$ is an ANR (an absolute retract (AR) if $M$ is an AR) (see~\cite[Theorem 4.1.6]{M}).
\end{itemize}

A space $X$ has the \emph{strong discrete approximation property} (SDAP) if any  map $f:\bigoplus_{n\in\N} \I^n\to X$ from the topological sum of finite-dimensional cubes can be approximated  arbitrarily closely by maps $g:\bigoplus_{n\in\N} \I^n\to X$ such that the family $\{g(\I^n): n\in\N\}$ is discrete.

\begin{fact}\cite[Proposition 1.7]{BM},~\cite[1.4.1.]{BRZ}\label{f7}
If $M$ is an ANR with SDAP or $M$ is locally compact then every $Z$-set in $M$ is a strong  $Z$-set in $M$.
\end{fact}

We will also need
\begin{fact}\label{f8}
If $X$ is a homotopy dense subset of a locally compact ANR $M$ and there are $Z$-sets $Z_i$ in $M$ such that $X\subset \bigcup_{i\in\N}Z_i$, then
$X$ has SDAP.
\end{fact}
The above fact can be easily derived from Facts~\ref{f7},~\ref{f6} and~\cite[Theorem 1.4.10]{BRZ} which says that each ANR $X$ that can be represented as a union of countably many strong $Z$-sets in $X$ has SDAP.

\

The famous Toruńczyk's theorem~\cite{Tor1} says that a  space $X$ is an $\\\R^\w$-manifold if and only if $X$ is a Polish ANR with SDAP.

The following theorem was proved by the first author~\cite{BRZ}.
\begin{theorem}\label{B1}
A space $X$ is an ANR with SDAP if and only if $X$ is homeomorphic to a homotopy dense subset of an $\R^\w$-manifold.
\end{theorem}

Let $\mathcal C$ be a topological class    of spaces. A space $X$ is \emph{strongly}  $\mathcal C$-\emph{universal} if for each $C\in \mathcal C$ and closed $B\subset C$, every map $f:C\to X$ which is a $Z$-embedding on $B$ can be approximated  arbitrarily closely by   $Z$-embeddings $g:C\to X$ such that $g\upharpoonright B= f\upharpoonright B$.

A space $X$ is called $\mathcal C$-\emph{absorbing} if
\begin{itemize}
\item
$X$ is an ANR with SDAP,
\item
$X=\bigcup_{n\in \N} X_n$, where each $X_n$ is a $Z$-set in $X$ and $X_n\in \C$,
\item
$X$ is strongly  $\mathcal C$-universal.
\end{itemize}

A fundamental theorem of M. Bestvina and J. Mogilski~\cite{BM} says that a $\C$-absorbing space is topologically unique up to a homotopy type. In particular,

\begin{theorem}\label{t:BM}
Any two $\C$-absorbing AR's are homeomorphic.
\end{theorem}

It is often more convenient to consider strongly universal pairs and absorbing pairs of spaces.

\textbf{From now on,    $\vec{\C}$ will denote a class of pairs $(K,C)$ such that $K$ is compact,  $C\subset K$ and $C\in \C$.}

A pair of spaces $(M,X)$ ($X\subset M$) is called
\begin{itemize}
\item
 \emph{strongly} $\vec{\C}$-\emph{universal} (some authors prefer to say $X$ is \emph{strongly} $\C$-\emph{universal in} $M$) if for each pair $(K,C)\in\vec{\C}$ and each closed $B\subset K$ every map $f:K\to M$ which is a $Z$-embedding on $B$ and satisfies $(f\upharpoonright B)^{-1}(X)=B\cap C$ can be approximated  arbitrarily closely by   $Z$-embeddings $g:K\to M$ such that $g\upharpoonright B= f\upharpoonright B$ and $g^{-1}(X)=C$.
\end{itemize}

\begin{remarks}\label{rem} In the above definition,
\begin{enumerate}
\item
if $M$ is an ANR, then  pairs $(K,C)$ can be replaced  by pairs $(\I^\w,C)$~\cite[Proposition 3.3]{BGM} ;
\item
if $M$ is an $\R^\w$- or $\I^\w$-manifold, then map $f$ can be replaced by an embedding~\cite[1.1.21, 1.1.26]{BRZ}
\end{enumerate}
\end{remarks}

Proving strong universality of pairs is usually cumbersome. An easier property is the \emph{preuniversality} which is  verified  as  a first step.

A pair $(M,X)$ is
\begin{itemize}
\item
$\vec{\C}$-\emph{preuniversal} if for any pair $(K,C)\in\vec{\C}$ there exists a map $f:K\to M$ such that $f^{-1}(X)=C$;
\item {\em everywhere $\vec{\C}$-preuniversal} if for any nonempty open set $U\subseteq X$ and pair $(K,C)\in\vec{\C}$ there exists a map $f:K\to M$ such that $f^{-1}(X)=C$.
\end{itemize}

\textbf{Henceforth, we  restrict our attention to  a Borel or projective class $\C\neq \mathbf \Pi^0_2$ containing all compacta}.

We gather several  general facts on strongly  $\vec{\C}$-universal pairs.
\begin{fact}\emph{(}\cite[Corollary 4.4]{BGM}\label{fact}
If  $M$ is an ANR (AR) and $(M,X)$ is  strongly $\vec{\C}$-universal, then $X$ and $M\setminus X$ are homotopy dense  in $M$ ANR's (AR's).
\end{fact}
\begin{fact}\cite[Corollary 6.2]{BGM}\label{fact1}
If  $M$ is an ANR, $Y\subset M$ is homotopy dense in $M$ and   $(Y,X)$ is  strongly $\vec{\C}$-universal, then  $(M,X)$ is  strongly $\vec{\C}$-universal.
\end{fact}
\begin{fact}\cite[Lemma 7.1]{BGM}\label{fact2}
If $M$ is an ANR, $(M,X)$ is  strongly $\vec{\C}$-universal and $U$ is an nonempty open subset of $E$, then $(U,X\cap U)$ is  strongly $\vec{\C}$-universal.
\end{fact}
\begin{fact}\cite[Proposition 7.2]{BGM}\label{fact3}
If  $M$ is an ANR, $\mathcal U$ is an open cover of $M$ and $(U,X\cap U)$ is  strongly $\vec{\C}$-universal for every $U\in \mathcal U$, then $(M,X)$ is  strongly $\vec{\C}$-universal.
\end{fact}
\begin{fact}\cite[Theorem 9.5]{BGM}\label{fact4}
 If  $M$ is an ANR, $(M,X)$ is  strongly $\vec{\C}$-universal and $A$ is a $Z$-set in $M$, then $(M,X\cup B)$ is strongly $\vec{\C}$-universal for every subset $B\subset A$.
\end{fact}

\begin{fact}~\cite[Theorem 3.1]{BC}\label{f5}
Suppose $M$ is an ANR, a subset $X\subset M$ has SDAP, $X$ is homotopy dense in $M$ and the pair $(M,X)$ is strongly $\vec{\C}$-absorbing. Then $X$ is strongly $\C$-absorbing.
\end{fact}

A pair $(M,X)$ is $\vec{\C}$-\emph{absorbing} (or $X$ is a $\C$-\emph{absorber in} $M$), if
\begin{itemize}
\item
$X\in \C$,
\item
$(M,X)$ is strongly $\vec{\C}$-universal,
\item
$X$ is contained in a $\sigma$-compact $\sigma Z$-set in $M$.
\end{itemize}

\

A fundamental theorem on absorbing pairs is the following.
\begin{theorem}~\cite[Corollary 10.8]{BGM}\label{t:fund}
If $M_i$ is an $\R^\w$- or $\I^\w$-manifold and pairs $(M_i,X_i)$ are $\vec{\C}$-\emph{absorbing}, $i=1,2$, then $X_1\cong X_2$ if and only if $X_1$ and $X_2$ are homotopically equivalent; in particular, if $X_1$ and $X_2$ are AR's, then $X_1\cong X_2$. If $M_i$ is an AR for $i=1,2$, then  $(M_1,X_1)\cong (M_2,X_2)$ under a homeomorphism $h$ of pairs (i.e., $h(M_1)=M_2$ and $h(X_1)=X_2$).
\end{theorem}

 Standard $\overrightarrow{\mathbf \Pi^0_3}$-absorbing pairs are $(\R^\w, c_0)$ and $(\I^\w,\hat{c_0})$, where
$\hat{c_0}=c_0\cap \I^\w$.
\cite{DMM}.
 More examples of $\overrightarrow{\mathbf \Pi^0_3}$-absorbing pairs can be found in~\cite{C2,CDGM,DMM,DR1, GM, GM1, KS, M}.

The Hurewicz set $\mathcal{H}(\I)$ is a $\mathbf \Pi^1_1$-absorber in $\K(\I)$~\cite[1.4.]{C1}.

\

\section{Strongly universal sets in Lawson semilattices}
A {\em topological semilattice} is a topological space $X$ endowed with a continuous commutative, associative operation $*:X\times X\to X$ such that $x*x=x$ for all $x\in X$.

For subsets $A$, $B$ of a semilattice $X$ denote $A*B:=\{a*b: a\in A, b\in B\}$. A subset $A$ of $X$ is a \emph{subsemillatice} if  $A*A\subset A$.
  A topological semilattice $X$ is called {\em Lawson} if it has a base of the topology consisting of subsemilattices.
A subsemilattice $A$ of  $X$ is a \emph{coideal} if $(X\setminus A)*X\subset X\setminus A$.

\begin{examples}\label{ex:L}

\begin{enumerate}
Natural examples of Lawson semilattices are
\item
Euclidean or Hilbert cubes with $(x_i)*(y_i):=(\max\{x_i,y_i\})$,
\item
 Vietoris hyperspaces $\K(X)$ with $A*B:=A\cup B$.
 \item
  Hyperspaces $\F(X)$, $\A_\omega(X)$, $\A_{\w+1}(X)$ and $\mathcal{H}(X)$  are coideals in $\K(X)$.
 \end{enumerate}
\end{examples}

\

A subset $X$ of a  space $M$ is called {\em locally path-connected in $M$} if for any point $x\in M$ and neighborhood $U_x\subseteq M$ of $x$ there exists a neighborhood $V_x\subseteq M$ of $x$ such that for any points $y,z\in V_x\cap X$ there exists a continuous map $\gamma:[0,1]\to U_x\cap X$ such that $\gamma(0)=y$ and $\gamma(1)=z$. Locally path-connected in $M$ subsets $X$ are also called  $LC^0$ in $M$.

A  space $X$ is {\em locally path-connected} ($LC^0$) if $X$ is $LC^0$ in $X$. If $X$ is locally path-connected in $M$, then $X$ is locally path-connected  but not conversely.

\

The following useful result was proved by W. Kubi\'s, K. Sakai and M. Yaguchi in~\cite{KSY}.

\begin{theorem}\label{t:KSY} If $X$ is a dense locally path-conected (and connected) subsemilattice in a Lawson semilattice $M$, then $M$ and $X$ are ANR's (AR's) and $X$ is homotopy dense in $M$.
\end{theorem}

The next theorem is an important  special case of  a more general result recently proved by the first author~\cite[Theorem 9]{B2}.

\begin{theorem}\label{t:B2}
Let $M$ be a Lawson semilattice and $X$ be a dense  in $M$ coideal which is $LC^0$ in $M$. If our  class $\C$ is $\mathbf \Pi^0_2$-hereditary (i.e. for each $C\in\C$, any $G_\delta$-subset of $C$ belongs to $\C$), then the following conditions are equivalent:
\begin{enumerate}
\item
the pair $(M,X)$ is strongly  $\vec{\C}$-universal;
\item
$(M,X)$ is everywhere  $\vec{\C}$-preuniversal.

\noindent  If $M$ is a Polish space and $X$ has SDAP, then conditions (1) and (2) are equivalent to
\item
$X$ is strongly $\C$-universal.
\end{enumerate}
\end{theorem}

\section{Two standard Borel absorbers in $\I^\omega$}

Consider the following  subsets of $\I^\w$:
$$
\begin{aligned}
&\Sigma_2=\{(x_i)_{i\in\w}\in \I^\w:\exists n\;\forall m\ge n\;\;(x_m=0)\};\\
&\Pi_3=\{(x_i)_{i\in\w}\in \I^\w:\forall n\;\exists k\;\forall m\ge k\;\;(x_{2^n(2m+1)}=0)\};\\
&\Sigma_4=\{(x_i)_{i\in\w}\in \I^\w:\exists n_0\;\forall n\ge n_0\;\exists m_0\;\forall m\ge m_0\;\;(x_{2^n(2m+1)}=0)\}.
\end{aligned}
$$

The sets $\Pi_3$ and $\Sigma_4$ belong to classes $\Pi_3$ and $\mathbf\Sigma^0_4 $, respectively, and are connected  analogs of $P_3$ and $S_4$ used in  the proof of Theorem~\ref{p1}.

As an application of Theorem~\ref{t:B2}, we are  going to show that the pairs  $(\I^\w,\Pi_3)$, $(\I^\w,\Sigma_4)$ are
absorbing for Borel classes            $\mathbf\Pi^0_3$ and $\mathbf\Sigma^0_4 $, respectively.

\begin{lemma}\label{l:1}
\begin{enumerate}
\item The pair $(\I^\w,\Sigma_2)$ is everywhere $\overrightarrow{\mathbf\Sigma^0_2}$-preuniversal;
\item The pair $(\I^\w,\Pi_3)$ is everywhere $\overrightarrow{\mathbf\Pi^0_3}$-preuniversal;
\item The pair $(\I^\w,\Sigma_4)$ is everywhere $\overrightarrow{\mathbf\Sigma^0_4 }$-preuniversal.
\end{enumerate}
\end{lemma}

\begin{proof} We will first prove that the pairs are preuniversal.
 For every $m\in\w$ let $\pr_m:\I^\w\to \I$, $\pr_m:x\mapsto x(m)$, be the coordinate projection.
\smallskip

1. Given a compact metrizable space $K$ and an $F_\sigma$-set $C\subset K$, write $C$ as the union $C=\bigcup_{n\in\w}C_n$ of an increasing sequence of closed sets $C_n$ in $K$. For every $n\in\w$ choose a continuous function $f_n:K\to\I$ such that $f_n^{-1}(0)=C_n$. Consider the diagonal product $f=(f_n)_{n\in\w}:K\to \I^\w$ and observe that $f^{-1}(\Sigma_2)=\bigcup_{n\in\w}C_n=C$.
\smallskip

2. Given a compact metrizable space $K$ and an $F_{\sigma\delta}$-set $C\subset K$, write $C$ as the intersection $C=\bigcap_{n\in\w}C_n$ of a decreasing sequence $(C_n)_{n\in\w}$ of $F_\sigma$-sets in $K$. By the preceding item, for every $n\in\w$ there exists a continuous map $f_n:K\to\I^\w$ such that $f_n^{-1}(\Sigma_2)=C_n$. Let $g_0:K\to\{0\}\subset\I$ be the constant function and, for every $k\in\N$, let $g_k=\pr_m\circ f_n$ where $n,m\in\w$ are unique numbers such that $k=2^n(2m+1)$.   Consider the diagonal product $g=(g_k)_{k\in\w}:K\to\I^\w$ and observe that $g^{-1}(\Pi_3)=\bigcap_{n\in\w}C_n=C$.
\smallskip

3. Let $C$ be any $F_{\sigma\delta\sigma}$-set in a compact metrizable space $K$.  By \cite[23.5(i)]{Ke}, there exists a sequence $(C_n)_{n\in\w}$ of $F_\sigma$-sets $C_n$ in $K$ such that $C=\bigcup_{m\in\w}\bigcap_{n=m}^\infty C_n$. By the first item, for every $n\in\w$ there exists a continuous function  $f_n:K\to \I^\w$ such that $f_n^{-1}(\Sigma_2)=C_n$.  Let $g_0:K\to\{0\}\subset\I$ be the constant function and, for every $k\in\N$, let $g_k=\pr_m\circ f_n$ where $n,m\in\w$ are unique numbers such that $k=2^n(2m+1)$.   Consider the diagonal product $g=(g_k)_{k\in\w}:K\to\I^\w$ and
observe that $g^{-1}(\Sigma_4)=\bigcup_{m\in\w}\bigcap_{n=m}^\infty C_n=C$.
\smallskip

In order to  see that the pairs are everywhere preuniversal, fix an open basic set $U=U_0\times\dots\times U_n\times \I\times\I\dots$ and apply an embedding $h:\I^\w\to U$ which is linear on each of the first $n+1$ coordinates and the identity on  the others. Observe that $h$ sends each of the sets    $\Sigma_2$, $\Pi_3$, $\Sigma_4$ into itself and use their preuniversality in $\I^\w$.

\end{proof}

Recall that $\I^\w$ is a Lawson semilattice (Examples~\ref{ex:L}).
\begin{lemma}\label{l:2}
The sets    $\Sigma_2$, $\Pi_3$, $\Sigma_4$ are dense coideals in $\I^\w$ and are $LC^0$ in $\I^\w$.
\end{lemma}
\begin{proof}
The first two properties are evident. Let $A\in\{\Sigma_2, \Pi_3, \Sigma_4\}$.  To see that $A$ is $LC^0$ in $\I^\w$, consider an open basic set $U=U_0\times\dots\times U_n\times \I\times\I\dots$   in $\I^\w$, where  $U_i$ is connected open in $\I$ for each $i\le n$, and choose arbitrary distinct points $(a_i)_{i\in\w}, (b_i)_{i\in\w}\in U\cap A$.
Denote $\Gamma=\{i:a_i\neq b_i\}$. There is a segment $\gamma(t)=(x_i(t))_{i\in\Gamma}\subset \I^{\Gamma}$ from $(a_i)_{i\in\Gamma}$ to  $(b_i)_{i\in\Gamma}$, $t\in\I$. Put $\bar{\gamma}(t)= (\bar{x}_i(t))_{i\in\w}$, where
$$\bar{x}_i(t)=\left\{
                 \begin{array}{ll}
                   x_i(t), & \hbox{\text{if $i\in\Gamma$};} \\
                   a_i=b_i, & \hbox{otherwise.}
                 \end{array}
               \right.$$
Then $\bar{\gamma}(t)$ is a segment from $(a_i)_{i\in\w}$ to $(b_i)_{i\in\w}$ in $U\cap A$.
\end{proof}

\begin{lemma}\label{l:3}
The sets    $\Sigma_2$, $\Pi_3$, $\Sigma_4$ are contained in a $\sigma$-compact $\sigma Z$-set in $\I^\w$.
\end{lemma}
\begin{proof}
This follows from the inclusions
$$ \Sigma_2\subset \Pi_3\subset \Sigma_4\subset \bigcup_{i\in\w} X_i,$$
where $X_i=\{(x_n)\in\I^\w: \text{$x_n=0$ for $n\le i$}\,\}$,
and from the fact that the pseudo-interior $(0,1)^\w$ is homotopy dense in the Hilbert cube $\I^\w$.
\end{proof}

Now,  Theorem~\ref{t:B2}, Lemmas~\ref{l:1},~\ref{l:2},~\ref{l:3} and Fact~\ref{f8} imply the following corollary.
\begin{corollary}\label{cor:1}
 \begin{enumerate}
\item
The pair $(\I^\w,\Pi_3)$ is $\overrightarrow{\mathbf\Pi^0_3}$-absorbing and $\Pi_3$ is $\mathbf\Pi^0_3$-absorbing.
\item
The pair $(\I^\w,\Sigma_4)$ is $\overrightarrow{\mathbf\Sigma^0_4}$-absorbing and $\Sigma_4$ is $\mathbf\Sigma^0_4$-absorbing.
\end{enumerate}
\end{corollary}

\section{A(N)R properties of $\A_n(X)$ ($n\le\w+1$) and $\mathcal H(X)$}

The following lemma is a slight generalization of~\cite[Lemma 3.2]{CuN}.
\begin{lemma}\label{l:4}
If a subspace $X\subset M$ is  $LC^0$ in $M$, then $\F(X)$ is   $LC^0$ in $\K(M)$.
\end{lemma}
\begin{proof}
 Let $U$ be an open in $M$ neighborhood of a point $x\in M$. There is an open in $M$ neighborhood $V\subset U$ of $x$ such that any two points $a,b\in V\cap X$ can be joined by a path in $U\cap X$.
\begin{claim}\label{c:1}
For each finite sets $A,B\subset V\cap X$ there is a path $$\gamma:\I\to \langle U\rangle \cap\F(X)\quad\text{such that $\gamma(0)=A$ and $\gamma(1)=B$}.$$
\end{claim}
Indeed, suppose $|A|\ge |B|$, choose a surjection $s:A\to B$ and paths $\gamma_{a,s(a)}:\I\to \langle U\rangle \cap\F(X)$ such that   $\gamma_{a,s(a)}(0)=a$ and $\gamma_{a,s(a)}(1)=s(a)$. Then $\gamma(t):=\bigcup_{a\in A}\gamma_{a,s(a)}(t)$ is  the required path.

\

Now, let $\langle U_1,\dots,U_k\rangle$ be a basic open set in the Vietoris topology in $\mathcal K(M)$ and $K\in \langle U_1,\dots,U_k\rangle$.

By the assumption and  compactness of $K$, one can find open in $M$ sets $V_1,\ldots, V_m$ such that
$K\in \langle V_1,\dots,V_m\rangle \subset \langle U_1,\dots,U_k\rangle,$ each $V_i$ is contained in some  $U_j$ and
and any two points $a,b\in V_i\cap X$ can be joined by a path in  $U_j\cap X$ for each $j$ such that $V_i\subset U_j$. Let $A,B\in \langle V_1,\dots,V_m\rangle \cap\F(X)$.
Using Claim~\ref{c:1}, one constructs inductively a path from $A$ to $B$ in $\langle U_1,\dots,U_k\rangle\cap \F(X)$.
\end{proof}

Clearly, if $X$ is dense in $M$ then $ \F(X)$ is  dense in  $\K(M)$ and if $X$ is connected, then $ \F(X)$ is connected either. Lemma~\ref{l:4} and Theorem~\ref{t:KSY} applied to $ \F(X)\subset \K(M)$ yield the following lemma.
\begin{lemma}\label{l:5}
If $X\subset M$ is dense and  $LC^0$ in $M$ (and connected), then $ \F(X)$ and $\K(X)$  are homotopy dense in  $\K(M)$ and the hyperspaces  $ \F(X)$, $\K(X)$ and $\K(M)$ are ANR's (AR's).
\end{lemma}

\begin{theorem}\label{t:6}
If a  subspace $X$ of a dense-in-itself space $M$ is dense and  $LC^0$ in $M$ (and connected), then
\begin{enumerate}
\item
$\A_\w(X)$, $\A_{\w+1}(X)$ and $\mathcal H(X)$ are  ANR's (AR's) which are $LC^0$ in $\K(M)$ and homotopy dense in  $\K(M)$;
\item
   $\A_n(X)$ is an ANR (AR) for each $n\in\N$.
\end{enumerate}
\end{theorem}

\begin{proof}
First, let us notice that for each $n\in\N\cup\{\w\}$  the hyperspace $\A_n(X)$ is dense in $\K(M)$. This follows easily from the fact that $\F(X)$ is dense in $\K(M)$, $M$ has no isolated points and there are nontrivial paths in $X$ in small neighborhoods of  points of $M$.

\smallskip

(1) Let $\mathcal U\subset \K(M)$ be an open neighborhood of $K\in \K(M)$. By Lemma~\ref{l:4}, there is an open $\mathcal V\subset \mathcal U$ containing $K$ such that any two $F_1,F_2\in  \F(X)\cap\mathcal V$ can be joined by a path in $\F(X)\cap\mathcal U$.

Let $A_1,A_2\in  \A_\w(X)\cap\mathcal V$. By Lemma~\ref{l:5}, there is a homotopy $$H:\K(M)\times\I\to \K(M),\quad H(Y,0)=Y,\quad H(Y,t)\in\F(X)$$ for each $Y$ and $t>0$.
Choose sufficiently small $0<t_0<1/2$  such that  $H(A_i,[0,t_0])\subset \mathcal V$, $i=1,2.$ Put $\gamma_1(t)=A_1\cup H(A_1,t)$ for $0\le t\le t_0$ and let $\gamma:[t_0,1-t_0]\to \F(X)\cap\mathcal U$ be a path such that $\gamma(t_0)=H(A_1,t_0)$ and $\gamma(1-t_0)=H(A_2,t_0)$.
Then $$\overline{\gamma}=\left\{
                           \begin{array}{ll}
                             \gamma_1(t), & \hbox{\text{if $0\le t\le t_0$};} \\
                             \gamma(t)\cup A_1, & \hbox{\text{if $t_0\le t\le 1- t_0$};} \\
                             H(A_2,t)\cup A_1, & \hbox{\text{if $1-t_0\le t\le 1$}}
                           \end{array}
                         \right.$$

is a path  in $\A_\w(X)\cap\mathcal U$ from $A_1$ to $A_1\cup A_2$. Similarly, there is a path in $\A_\w(X)\cap\mathcal U$ from $A_2$ to $A_1\cup A_2$. It means that the hyperspace $\A_\w(X)$ is $LC^0$ in $\K(M)$. The proof for $\A_{\w+1}(X)$ and $\mathcal H(X)$ is the same.

Since  $\A_\w(X)$ is a dense subsemilattice of $\K(M)$, we conclude by Theorem~\ref{t:KSY} that $\A_\w(X)$ is a homotopy dense in  $\K(M)$ ANR, which implies that also $\A_{\w+1}(X)$ and $\mathcal H(X)$ are homotopy dense ANR's.  If $X$ is connected then $\K(M)$ is AR (Lemma~\ref{l:5}), hence  all the hyperspaces are  AR's, as  homotopy dense subsets.

\smallskip

(2) Theorem~\ref{t:KSY} is not applicable to  hyperspaces $\A_n(X)$, $n\in\N$, for they are not subsemilattices of  $\K(M)$. Therefore, we provide a more direct argument.

Consider basic open sets $\langle U_0,U_1,\dots,U_k\rangle$ in the Vietoris topology in $\mathcal K(X)$, $k\ge 0,$ where $U_i$'s are open path-connected subsets of $X$.
 The sets  $\mathcal A_n(X)\cap \langle U_0,U_1,\dots,U_k\rangle$ form an open base in  $\mathcal A_n(X)$ which is closed under finite intersections. We are going to show  that each of them  is contractible in itself.

Choose  points $x_0\in U_0$ and $s_i\in U_i$ for each $0<i\le k$. Let  $$h:\I\to h(\I)\subset U_0\quad\text{be a homeomorphism such that $h(0)=x_0$}.$$  For $r\in \I$, let
$$C_r=h\left(\{0\}\cup \left\{\frac{r}{j}: j\in\mathbb N\right\}\right)\quad\text{and}\quad S=C_1\cup\{s_1,\dots, s_k\}.$$

The set $\F(X)\cap \langle U_0,\ldots,U_k\rangle$ is dense in the Lawson semilattice $\langle U_0,\ldots,U_k\rangle$ and it is $LC^0$ in  $\langle U_0,\ldots,U_k\rangle$, by Lemma~\ref{l:4}, so it is an ANR  homotopy dense in $\langle U_0,\ldots,U_k\rangle$, by Theorem~\ref{t:6}.
Hence, there is a homotopy $$H:\langle U_0,\ldots,U_k\rangle\times \I\to \langle U_0,\ldots,U_k\rangle\quad\text{through finite sets},$$ i.e.
$$H(K,0)=K \quad\text{and}\quad H(K,t)\in \mathcal F(X) \quad\text{for}\quad t>0.$$

 The subspace $U=\bigcup_{i=0}^k U_i$ is locally path-connected and $$\mathcal E=\mathcal F(X) \cap \langle U_0,U_1,\dots,U_k\rangle$$ is an expansion hyperspace in $U$ in the sense of~\cite{CuN}. Moreover, each element of $\mathcal E$ intersects each component of $U$.
Therefore $\mathcal E$ is an AR~\cite[Lemma 3.6]{CuN}, so it is contractible. Since  $\{x_0,s_1,\dots,s_k\}\in \mathcal E$, there is a homotopy $F: \mathcal E \times \I \to \mathcal E$ such that
$$F(Y,0)=Y \quad\text{and}\quad F(Y,1)=\{x_0,s_1,\dots,s_k\}.$$
Define a homotopy
$$G:\mathcal A_n(X)\cap \langle U_0,U_1,\dots,U_k\rangle \times \I\to \mathcal A_n(X)\cap \langle U_0,U_1,\dots,U_k\rangle,\quad\text{by}$$
\begin{multline*} G(Y,t)=\\
\left\{
  \begin{array}{ll}
    Y\cup H(Y,4t), & \hbox{for $t\in[0,1/4]$;} \\
    Y\cup F(H(Y,1),4(t-1/4)), & \hbox{for $t\in[1/4,1/2]$;} \\
    H(Y,4(t-1/2))\cup C_{4(t-1/2)}\cup \{s_1,\dots,s_k\}, & \hbox{for $t\in[1/2,3/4]$;} \\
    F(H(Y,1),4(t-3/4))\cup S, & \hbox{for $t\in[3/4,1]$.}
  \end{array}
\right.
\end{multline*}
Homotopy $G$ is a deformation which contracts $\mathcal A_n(X)\cap \langle U_0,U_1,\dots,U_k\rangle$ in itself to the point $S$.
In particular, if $X$ is connected then $\mathcal A_n(X)=\mathcal A_n(X)\cap \langle X\rangle$ is contractible.

Summarizing: $\mathcal A_n(X)$ is a  locally connected space with an open base closed under finite intersections, each of whose elements is connected and homotopically trivial. It means that $\mathcal A_n(X)$ is an ANR (see~\cite[Corollary 4.2.18]{M}); if $X$ is connected then $\mathcal A_n(X)$ is contractible, hence an AR.

\end{proof}

\section{Universality and absorbing properties of  $\A_\w(X)$, $\A_{\w+1}(X)$ and $\mathcal H(X)$}\label{s:abs}

\begin{lemma}\label{l:puI} The pair $(\K(\I),\A_\w(\I))$ is $\overrightarrow{\mathbf\Sigma^0_4 }$-preuniversal. The pairs $(\K(\I),\A_{\w+1}(\I))$ and
$(\K(\I),\mathcal H(\I))$ are $\overrightarrow{\mathbf\Pi^1_1}$-preuniversal.
\end{lemma}

\begin{proof} Consider the map
 $$\psi:\I^\w\to \K(\I),\quad \psi\bigl((x_n)_{n\in\w}\bigr) =
\cl\bigl(\{2^{-(n+1)}+2^{-(n+m+1)}x_{2^n(2m+1)}:n,m\in\w\}\bigr).$$
Observe that the preimage $\psi^{-1}(\A_\w(\I))=\Sigma_4$. Now the strong $\overrightarrow{\mathbf\Sigma^0_4 }$-universality of the pair $(\I^\w,\Sigma_4)$ (Corollary~\ref{cor:1}) implies the
 $\overrightarrow{\mathbf\Sigma^0_4 }$-preuniversality of the pair $(\K(\I),\A_\w(\I))$.

The preuniversality of $(\K(\I),\mathcal H(\I))$ follows directly from the R. Cauty's result that the pair is $\overrightarrow{\mathbf\Pi^1_1}$-absorbing~\cite{Cu1}. In fact,  the construction in~\cite{Cu1} shows that  $(\K(\I),\A_{\w+1}(\I))$ is strongly $\overrightarrow{\mathbf\Pi^1_1}$-universal.

\end{proof}

Lemma~\ref{l:puI} implies

\begin{lemma}\label{l:eu} Let  $X$ be a dense subspace of a space $M$ such that for any non-empty set $U\subset M$ the intersection $U\cap X$ contains a topological copy of the segment $\I$. Then the pair $(\K(M),\A_\w(X))$ is everywhere $\overrightarrow{\mathbf\Sigma^0_4}$-preuniversal and pairs $(\K(M),\A_{\w+1}(X))$ and
$(\K(M),\mathcal H(X))$ are everywhere $\overrightarrow{\mathbf\Pi^1_1}$-preuniversal.
\end{lemma}

\

\begin{theorem}\label{t:uni}
If  $X$ is a dense subset of a dense-in-itself space $M$ and $X$ is $LC^0$ in $M$, then
 the pair $(\K(M),\A_\w(X))$ is strongly $\overrightarrow{\mathbf\Sigma^0_4}$-universal. The pairs $(\K(M),\A_{\w+1}(X))$ and $(\K(M),\mathcal H(X))$ are strongly $\overrightarrow{\mathbf\Pi^1_1}$-universal.
\end{theorem}

\begin{proof}
Recall that $\A_\w(X)$, $\A_{\w+1}(X)$ and $\mathcal H(X)$  are  dense coideals in the Lawson semilattice  $\K(M)$ (Examples~\ref{ex:L}) which  is $LC^0$ in $\K(M)$ (Theorem~\ref{t:6}).  In view of Lemma~\ref{l:eu},  the conclusion follows from Theorem~\ref{t:B2}.
\end{proof}

\

\begin{theorem}\label{t:abs}
Let $X$ be a dense $F_\sigma$-subset of a dense-in-itself   Polish space $M$ and assume $X$ is $LC^0$ in $M$.
\begin{enumerate}
\item If $M$ is locally compact, then  the pair $(\K(M),\A_\w(X))$ is  $\overrightarrow{\mathbf\Sigma^0_4}$-absorbing. The pairs $(\K(M),\A_{\w+1}(X))$ and $(\K(M),\mathcal H(X))$ are  $\overrightarrow{\mathbf\Pi^1_1}$-absorbing.
\item
If $M$ is $LC^0$ and nowhere locally compact (i.e., no point has a compact neighborhood), then $\A_\w(X)$ is $\mathbf\Sigma^0_4$-absorbing and sets $\A_{\w+1}(X)$, $\mathcal H(X)$ are $\mathbf\Pi^1_1$-absorbing.
\end{enumerate}

\end{theorem}

\begin{proof}
By Corollary~\ref{Polish}, the hyperspace $\A_\w(X)$ is an  $F_{\sigma\delta\sigma}$-subset of the Polish space $\K(M)$, hence $\A_{\w}(X)\in \mathbf\Sigma^0_4 $. By Theorem~\ref{t:uni},  $(\K(M),\A_\w(X))$ is strongly $\overrightarrow{\mathbf\Sigma^0_4}$-universal. The spaces $\A_{\w+1}(X)$ and $\mathcal H(X)$ belong to class $\mathbf\Pi^1_1$ if $X$ is Polish; if not, then by~\cite[(33.5)]{Ke} $\K(X)$ is in  $\mathbf\Pi^1_1$ and   the hyperspaces $\A_{\w+1}(X)= \A_{\w+1}(M) \cap \K(X)$ and $\mathcal H(X)= \mathcal H(M)\cap \K(X)$ also belong to  $\mathbf\Pi^1_1$ as intersections of a Polish space and $\mathbf\Pi^1_1$-sets.

Notice that $\A_\w(X)\subset \A_{\w+1}(X)\subset \mathcal H(X)$. Therefore in case (1), it remains to find  a $\sigma$-compact $\sigma Z$-subset of $\K(M)$ that covers $\mathcal H(X)$.
 We may use the following idea due to R. Cauty~\cite[Lemme 5.6]{C1}. We can assume  that the metric $\rho$  in $M$ is bounded by 1. The locally path-connected space $X$ admits an equivalent metric
$$d(x,y)=\left\{
           \begin{array}{ll}
             \inf\{\diam_\rho(C): \text{$C$ is a continuum in $X$ containing $x,y$}\},\\
 \hbox{\text{if such continuum $C$ exists};} \\
             1, \hbox{otherwise.}
           \end{array}
         \right.
$$
Define
$$
Z_k=\{K\in \K(M): |K|\ge 2\quad\text{and}\quad (\exists x\in K)\, d(x,K\setminus \{x\})\ge\frac{1}{k}\}.$$
Each $Z_k$ is a closed subset of $\K(M)$, thus it is $\sigma$-compact. In order to show that the sets  are  $Z$-sets in $\K(M)$, we apply to them a deformation $H_t:\K(M)\to \K(M)$ through $\K(X)$ (it exists by Lemma~\ref{l:5}) followed by  the, so called, expansion deformation
$$E_t:\K(X) \to  \K(X), \quad E_t(K)=\{x\in X: d(x,K)\le t\}.$$
More precisely, for $t>0$ take a map $f_t=E_t\circ H_t:\K(M)\to \K(X)$.
Each $Z_k$  contains  an isolated point, while  $E_t(K)$ for $t>0$ has no isolated points. Hence,  for   sufficiently small $t>0$,
$f_t$
maps $\K(M)$ into  $\K(M)\setminus \bigcup_{k\in\N} Z_k$ and   approximates the identity map on $\K(M)$. Moreover, since each  $A\in \mathcal H(X)$ contains an isolated point, we get the desired inclusion $\mathcal H(X)\subset \bigcup_{k\in\N} Z_k$.

In case (2),  $\K(M)$ is an $\R^\w$-manifold (see~\cite{Cu3}) and $\A_\w(X),$ $\A_{\w+1}(X),$ $\mathcal H(X)$ being homotopy dense in $\K(M)$, they are ANR's with SDAP by Theorem~\ref{B1}. Moreover, since the pairs $(\K(M),\A_\w(X))$, $(\K(M),\A_{\w+1}(X))$ and $(\K(M),\mathcal H(X))$ are strongly universal in respective classes~\ref{t:uni}, the spaces $ \A_\w(X)$, $\A_{\w+1}(X)$ and $\mathcal H(X)$  are universal in the classes by~\ref{f5}.

Since $Z$-sets in the $\R^\w$-manifold $\K(M)$ are strong $Z$-sets (Fact~\ref{f7}),  the sets $Z_k\cap \A_\w(X)$, $Z_k\cap\A_{\w+1}(X)$ are $Z$-sets in $\A_\w(X)$, $\A_{\w+1}(X)$ and $\mathcal H(X)$, respectively,  by Theorem~\ref{t:6} and Fact~\ref{f6}.   They also belong to the classes. Therefore all sufficient conditions for absorbing sets in the classes  are satisfied.

\end{proof}

Finally, we get the following characterzations.
\begin{theorem}\label{t:abs1}
$\A_\w(X)\cong \Sigma_4$, $A_{\w+1}(X)\cong \mathcal H(X)\cong \mathcal H(\I)$  in each of the following cases.
\begin{enumerate}
\item
 $X$ is  nondegenerate, connected, locally connected and locally compact (i.e. $X$ is a nondegenerate generalized Peano continuum); if $X$ is compact (i.e., $X$ is a nondegenerate Peano continuum), then
\begin{itemize}
\item
$(\K(M),\A_\w(X))\cong (\I^\w,\Sigma_4)$,
\item
$(\K(M),\A_{\w+1}(X))\cong (\K(M),\mathcal H(X))\cong(\K(\I),\mathcal H(\I))$.
\end{itemize}
\item
$X$ is  nondegenerate, Polish, connected, locally connected and nowhere locally compact.
\end{enumerate}
\end{theorem}

\begin{proof}
(1) for noncompact $X$ follows from Theorem~\ref{t:abs}, Corollary~\ref{cor:1}, Theorem~\ref{t:6} and Theorem~\ref{t:fund}. In case when $X$ is a nondegenerate Peano continuum, we also use the Curtis-Schori characterization  $\K(M)\cong \I^\w$~\cite{Cu1}.

(2) follows from  Theorem~\ref{t:abs}, Corollary~\ref{cor:1}, Theorem~\ref{t:6} and Theorem~\ref{t:BM}.

\end{proof}

\section{Hyperspaces $\A_n(\I)$, $n\in\N$}\label{interval}

Recall that $\K(\I)\cong \I^\w$. It will be more convenient to work with the pair $(\K(\J),\A_n(\J))$, where $\J=[-1,1]$.

\begin{lemma}\label{l:A1}
The pair $(\K(\J),\A_n(\J))$ is strongly $\overrightarrow{\mathbf\Pi^0_3}$-universal.
\end{lemma}

\begin{proof}
We apply an approach developed in~\cite{GM}.

Let $C$ be an $F_{\sigma\delta}$-subset of  $\I^\w$,  $B$ a closed subset of $\I^\w$, $f:\I^\w\to \K(\J)$ an embedding which is a $Z$-embedding on $B$ and $\epsilon>0$.
Our goal is to find a $Z$-embedding $g:\I^\w\to \K(\J)$ such that $g\upharpoonright B=f\upharpoonright B$, $g^{-1}(\A_n(\J))\setminus B=C\setminus B$, and $\dist (f(x),g(x))<\epsilon$ for each $x\in \I^\w$, where $\dist$ denotes the Hausdorff distance in the hyperspace $\K(\J)$
(see Remarks~\ref{rem}).

The first ingredient in a construction of an approximation $g$ is the embedding $\phi_n:\I^w\to \K(\J),$
\begin{multline}\label{e:phi}
 \phi_n\bigl((x_j)_{j\in\w}\bigr)= \{-1\}\cup \{-(2^{-(j+1)}+x_j2^{-(j+2)}): j\in\w\}\cup \chi(n) \cup\\
\cl\bigl(\bigl\{2^{-(j+1)}+2^{-(j+k+1)}x_{2^{j-n}(2k+1)}:j\ge n, k\in\w\bigr\}\bigr),
\end{multline}
where $\chi(n)$ is defined in~\eqref{e:chi}. The positive part of $ \phi_n\bigl((x_j)_{j\in\w}\bigr)$  is responsible for the property $\phi_n^{-1}(\mathcal A_n(\J))=\Pi_3,$
while the negative one  exhibits 1-1 correspondence $\phi_n: \I^\w \to \K(\J)$.

The pair $(\I^\w,\Pi_3)$ being strongly $\overrightarrow{\mathbf\Pi^0_3}$-universal, there is an embedding $\zeta:\I^\w\to \I^\w$ such that $\zeta^{-1}(\Pi_3)=C$. Put
\begin{equation}\label{eq:xi}
\xi=\psi\zeta.
\end{equation}
 We have
 \begin{equation}\label{eq1}
\xi^{-1}(\mathcal A_n(\J))=C.
\end{equation}
Next, we need a deformation $H:\K(\J)\times \I\to \K(\J)$ through finite sets. Deformation $H$ can be easily modified to satisfy
$$\dist\bigl(K,H(K,t)\bigr)\le 2t\quad\text{and}\quad H(K,t)\subset [-1+t,1-t]$$
(see~\cite[(1-4), p. 183]{GM}).

We are going to verify that the embedding  $g:\I^\w\to \K(\J)$  defined by
\begin{equation}\label{eq:g}
g(x)= H(f(x),\mu(x))\ \cup \ \bigl(\min H(f(x),\mu(x))+ \mu(x)\xi(x)\bigr),
\end{equation}
 where
\begin{equation}\label{eq:mu}
\mu(x) =\frac1{4}\min\{\epsilon, \min\{\dist(f(x),f(z)): z\in B\}\}
\end{equation}
(we use standard operations $\alpha A:= \{\alpha a: a\in A\}$ and $x+ A:=\{x+a: a\in A\}$),
satisfies the definition of strong  $\overrightarrow{\mathbf\Pi^0_3}$-universality of pair $(\K(\J),\mathcal A_n(\J))$.

Clearly, $g$ is continuous and since $\mu(x)=0$ for $x\in B$,  it agrees with $f$ on $B$. It also $\epsilon$-approximates $f$, as
\begin{multline}\label{estim}
\dist(f(x),g(x))\le \\
\dist\bigl(f(x),H(f(x),\mu(x))\bigr) + \dist\bigl(H(f(x),\mu(x)),g(x)\bigr) \le \\ 3\mu(x)=
\frac34\min\{\epsilon, \min\{\dist(f(x),f(z)): z\in B\}\}.
  \end{multline}
Mapping $g$ is 1-1 on $B$. So, let $x,y\in \I^\w \setminus B$. Then both numbers $\mu(x)$ and   $\mu(y)$ are positive. Suppose $g(x)=g(y)$. It follows that
\begin{equation}\label{eq:min}
-\mu(x)=\min g(x)
= \min g(y) = -\mu(y)
\end{equation}
 hence $\mu(x)=\mu(y)$ and $\zeta (x)_j=\zeta (y)_j$ for each $j\in\w$, so $x=y$.
Suppose $x\in B$,  $y\notin B$ and $g(x)=g(y)$. Then $g(y)=f(x)$. On the other hand,
\begin{multline*}
\dist(g(y),f(B))\ge \dist(f(y),f(B))-\dist(f(y),g(y))\ge\\ \frac14\dist(f(y),f(B))>0,\end{multline*}
by~\eqref{estim} and since $f$ is 1-1, a contradiction.
Thus, $g$ is 1-1.

It follows from~\eqref{eq1} that  $g^{-1}(\mathcal A_n(\J))\setminus B=C\setminus B.$

The image  $g(\I^\w)$ is a $Z$-set in $\K(\J)$. Indeed, $g(\I^\w\setminus B)=g( \I^\w)\setminus g(B)$ is a $\sigma Z$-set in $\K(\J)$ because
 deformation $H$ through finite sets
 satisfies
$$\forall(t>0) \ H(g(\K(\J)\setminus B),t) \cap g( \K(\J)\setminus B)=\emptyset$$
(since, for each $x\notin B$ and $t>0$, the set $H(g(x),t)$ is  finite whereas $g(x)$ is not).  Now, $g(B)$ is a $Z$-set, so the union  $g(B)\cup g( \K(\J)\setminus B)=g( \K(\J))$ is a compact $\sigma Z$-set, hence a $Z$-set in $ \K(\J)$.

\end{proof}

\begin{lemma}\label{l:Zset}
$\A_n(\J)$ is contained in a $\sigma Z$-set in $ \K(\J)$.
\end{lemma}
\begin{proof}
The $\sigma Z$-set we are looking for was constructed in a more general setting (for generalized Peano continua) in the proof of Case (1) of Theorem~\ref{t:abs}.
\end{proof}

Since $\A_n(\I))$ is in class $\mathbf\Pi^0_3$, Lemmas~\ref{l:A1} and~\ref{l:Zset} imply
\begin{theorem}\label{t:An}
The pair  $(\K(\I),\A_n(\I))$ is $\overrightarrow{\mathbf\Pi^0_3}$-absorbing for each $n\in\N$. Consequently, $\A_n(\I)\cong \Pi_3\cong c_0\cong\hat{c_0}$.
\end{theorem}

\begin{corollary}\label{co1}
$\mathcal A_n((0,1))$, $\mathcal A_n([0,1))$ and  $\mathcal A_n((0,1])$  are also $\mathbf\Pi^0_3$-absorbers in $\K(\I)$ for each $n\in\mathbb N$. Hence they are all homeomorphic to $c_0$.
\end{corollary}

\begin{proof}
Observe that the set  $\mathcal B= \{A\in \K(\I): A\cap \{0,1\}\neq \emptyset\}$ is a $Z$-set in $\K(\I)$ and  $\mathcal A_n((0,1))=\mathcal A_n(\I)\setminus \mathcal B$.
It is known from~\cite[Corollary 9.4]{BGM}  that the difference of an $\mathbf\Pi^0_3$-absorber and a $Z$-set in a Hilbert cube $\K(\I)$ is again an $\mathbf\Pi^0_3$-absorber in $\K(\I)$. The argument for the remaining intervals is similar.
\end{proof}

\begin{remark}
Corollary~\ref{co1} absorbs~\cite[Theorem 2.4, Corollary 2.5]{GO}, \cite[Theorem 5.1]{CMP}
and provides a positive answer to the question in~\cite[Question 2.17]{GO} of whether or not $\mathcal S_c([0,1])$ is homeomorphic to $\mathcal S_c((0,1))$.
\end{remark}

\begin{remark}\label{rem0}
The above method of showing the strong  $\overrightarrow{\mathbf\Pi^0_3}$-universality is specific for $X$ an arc---we continuously select a point from a  finite set (the point $\min H(f(x),\mu(x))$ from $H(f(x),\mu(x))$) and such  selections are characteristic for arcs~\cite{KNY}.
Using Corollary~\ref{co1} and general facts about strongly $\overrightarrow{\mathcal \C}$-univer\-sal pairs, it is  shown in Section~\ref{spheres} that  the pair $(\mathcal K(S^1),\mathcal A_n(S^1))$ is  $\overrightarrow{\mathbf\Pi^0_3}$-absorbing.
One may ask for what other ``nice'' spaces $X$ the pair $(\mathcal K(X),\A_n(X))$ is $\overrightarrow{\mathbf\Pi^0_3}$-absorbing.
 \end{remark}

\section{Hyperspaces $\mathcal A_n(S^1)$, $n\in\N$}\label{spheres}
By  $S^1$ we denote the unit circle in $\R^2$.
\begin{theorem}\label{circle}
The pair $(\mathcal K(S^1),\mathcal A_n(S^1))$ is $\overrightarrow{\mathbf\Pi^0_3}$-absorbing. Hence, $\mathcal A_n(S^1)\cong c_0$.
        \end{theorem}
\begin{proof}
The hyperspace $E=\mathcal K(S^m)\setminus \{S^1\}$ is an AR. Let $\mathcal U$ be its open cover by  sets $U_p=\mathcal K(S^1\setminus \{p\})$, $p\in S^1$. The pair $(\mathcal K(\I),\mathcal A_n((0,1)^m)$ is  strongly $\overrightarrow{\mathbf\Pi^0_3}$-universal by Corollary~\ref{co1}. By Fact~\ref{fact2},the pair  $(\mathcal K((-1,1)),\mathcal A_n((0,1)))$ is  strongly $\overrightarrow{\mathbf\Pi^0_3}$-universal. Let $h:(0,1) \to  S^1\setminus \{p\}$ be a homeomorphism and $\tilde{h}: \mathcal K((0,1)) \to \mathcal K(S^1\setminus \{p\})$ be the induced homeomorphism. Clearly, $\tilde{h}( \mathcal A_n((0,1)))=\mathcal A_n(S^1\setminus \{p\})$ and the pair $(U_p,\mathcal A_n(S^1\setminus \{p\}))$ is strongly $\overrightarrow{\mathbf\Pi^0_3}$-universal. Since
$\mathcal A_n(S^1\setminus \{p\})= U_p \cap \mathcal A_n(S^1)$, we  infer by Fact~\ref{fact3} that the pair $(E,\mathcal A_n(S^m))$ is    strongly $\overrightarrow{\mathbf\Pi^0_3}$-universal.
The singleton $\{S^1\}$ being a $Z$-set in  $\mathcal K(S^1)$, $E$ is homotopy dense in $\mathcal K(S^1)$. Then, by Fact~\ref{fact1},  the hyperspace  $(\mathcal K(S^1),\mathcal A_n(S^1))$ is  also  strongly $\overrightarrow{\mathbf\Pi^0_3}$-universal.

The proof that $\mathcal A_n(S^1)$ is contained in a $\sigma Z$-set in $\mathcal K(S^1)$ is the same as for Lemma~\ref{l:Zset}.

\end{proof}

\section{Hyperspaces $\A_n(X)^F$}\label{Peano}

As we have noticed in Remark~\ref{rem0}, there is an essential obstacle in proving that $\mathcal A_n(X)$ is  an $\mathbf \Pi^0_3$-absorber in $\mathcal K(X)$ for nondegenerate Peano continua other than $\I$ and $S^1$. The obstacle disappears for  hyperspaces $\mathcal A_n(X)^F\subset \K(X)^F$ and   $\mathcal A_1(X,\{p\})\subset \K(X)^{\{p\}}$, where $F$ is a fixed finite subset of $X$ which contains a point of order $\ge 2$ and   $p$ is a fixed point of order  $\ge 2$
(a point $p\in X$ is \emph{of order }$\ge 2$ if there is an arc $L\subset X$   containing $p$ in its combinatorial interior).
 The latter hyperspace is a natural counterpart of $c_0$ whose elements converge to the same number $0$.

\

\begin{theorem}\label{t2}
 Suppose $X$ is a  Peano  continuum, $F\subset X$ is  finite and contains a point $p$ of order $\ge 2$, $n\in \mathbb N$.   Then
the pairs $(\K(X)^F,   \mathcal A_n(X)^F)$ and $(\K(X)^{\{p\}}, \mathcal A_1(X,\{p\})$ are $\overrightarrow{\mathbf \Pi^0_3}$-absorbing.

Consequently,
$\mathcal A_n(X)^F\cong \mathcal A_1(X,\{p\})\cong c_0$.
 \end{theorem}

\begin{proof}
Recall that $\mathcal K(X)^F$ is a Hilbert cube~\cite{Cu1}.

Clearly, $\mathcal A_n(X)^F=\mathcal A_n(X)\cap \K(X)^F$ is $F_{\sigma\delta}$ in  $\K(X)^F$. Also $\mathcal A_1(X,\{p\})$ is   $F_{\sigma\delta}$ in $\K(X)^{\{p\}}$, since it  equals the preimage $D^{-1}(\{p\})$, where $D$ is the derived set operator on $\K(X)^{\{p\}}$.

In order to prove the strong $\overrightarrow{\mathbf \Pi^0_3}$-universality, we proceed similarly to the proof of Lemma~\ref{l:A1}.

There is a deformation $H:\mathcal K(X)\times [0,1]\to \mathcal K(X)$ through finite sets such that $\dist (H(A,t),A)\le 2t$.
If we add  $F$  to each $H(A,t)$ we get  a continuous deformation $\K(X)^F\times [0,1] \to \K(X)^F$ through finite sets satisfying  $\dist (H(A,t),A)\le 2t$ for $A\in \K(X)^F$.  So, we can assume that $H:\K(X)^F\times [0,1] \to \K(X)^F$ is such. Choose  an arc $L\subset X$ containing $p$ in its combinatorial interior.   Note that each set $H(A,t)$ contains $p$. To simplify further description, assume without loss of generality that  $L=\J=[-1,1]$ and $p=0$. We    modify the definition of embedding $\phi_n$ from~\eqref{e:phi} in its ``negative'' part in which the sequence $\{-(2^{-(j+1)}+x_j2^{-(j+2)}): j\in\w\}$ is now replaced with an increasing  sequence $l(x)$ obtained in the following way.
For any $x=(x_j)\in \I^\w$, put
$$a(x)_j= -\bigl(2^{2^j}+\frac{x_j}{2^{2^j}}\bigr)^{-1}.$$
Observe that the sequence $a(x)=(a(x)_j)$ satisfies
\begin{enumerate}
\item
$a(x)$ is strictly increasing and converging to $0$,
\item
for each $x,y\in \I^\w$ and $i<j$, vectors $$\bigl(a(x)_i, a(x)_{i+1}\bigr)\quad\text{and}\quad  \bigl(a(y)_j, a(y)_{j+1}\bigr)$$ are not parallel.
\end{enumerate}
Let  $x'\in \I^\w$ be the sequence $1, x_1,1, x_1,x_2,1, x_1,x_2,x_3,1, \ldots$. Put $l(x)=a(x')$. Clearly, $l(x)$ also satisfies conditions (1-2).

Now, let
\begin{multline}\label{psil}
    \psi_n(x)=  l(x)  \cup \chi(n) \cup\\
   \cl\bigl(\bigl\{2^{-(j+1)}+2^{-(j+k+1)}x_{2^{j-n}(2k+1)}:j\ge n, k\in\w\bigr\}\bigr)
   \end{multline}

  Note that, for each $n\in \mathbb N $, we have
 \begin{equation}\label{psilreduction}
\psi_n^{-1}\bigl(\mathcal A_n(\J)\bigr)=\Pi_3 = \psi_1^{-1}\bigl(\mathcal A_1(\J,\{0\})\bigr).
\end{equation}
  Define
\begin{equation}\label{eq:g2}
g(x)= H(f(x),\mu(x))\ \cup \ \mu(x)\xi(x).
\end{equation}
where $\mu(x)$ is defined in~\eqref{eq:mu} and  $\xi(x)$ is defined  by~\eqref{eq:xi} with $\psi_n$ modified as above.
Now, $l(x)$ is responsible for  $g(x)$ being 1-1. Indeed, suppose  $x,y\in \I^\w \setminus B$ and $g(x)=g(y)$. Then $\mu(x)$ and $\mu(y)$ are positive.  If the set $$W=\bigl(H(f(x),\mu(x)) \cup H(f(y),\mu(y))\bigr)\cap [-1,0)$$ is non-empty, let $\alpha=\max W$
and
 notice that, for sufficiently large $k$, say for $k\ge j$,  numbers $\mu(x)l(\zeta(x))_k$ and  $\mu(y)l(\zeta(y))_k$ are greater than  $\alpha$;
if $W=\emptyset$, then put $j=0$.
So, we can assume  that $\mu(x)l(\zeta(x))_j=\mu(y)l(\zeta(y))_i$ for some $i\ge j$. Since sequences $\mu(x)l(\zeta(x))$ and $\mu(y)l(\zeta(y))$ are increasing, it follows that also $\mu(x)l(\zeta(x))_{j+1}=\mu(y)l(\zeta(y))_{i+1}$. Thus $i=j$ by    property $(2)$. But then   $\mu(x)l(\zeta(x))_k=\mu(y)l(\zeta(y))_k$ for each $k\ge j$. Choose $m\ge j$ such that  $(\zeta(x)')_m=1=(\zeta(y)')_m$. Then
\begin{multline*}
-\mu(x)(2^{2^m}+\frac{1}{2^{2^m}})^{-1}= \mu(x)a(\zeta(x)')_m=\mu(x)l(\zeta(x))_m=\\
\mu(y)l(\zeta(y))_m=\mu(y)a(\zeta(y)')_m=-\mu(y)(2^{2^m}+\frac{1}{2^{2^m}})^{-1}
\end{multline*}
which implies
$\mu(x)=\mu(y)$. Hence, for each $k\ge j$,
\begin{multline*}
-\bigl(2^{2^k}+\frac{\zeta(x)'_k}{2^{2^k}}\bigr)^{-1}=a(\zeta(x)')_k=l(\zeta(x))_k=\\
l(\zeta(y))_k=a(\zeta(y)')_k= -\bigl(2^{2^k}+\frac{\zeta(y)'_k}{2^{2^k}}\bigr)^{-1},
\end{multline*}
so $\zeta(x)'_k=\zeta(y)'_k.$
 Consequently, $\zeta(x)=\zeta(y)$ and $x=y$.

The remaining arguments are exactly the same as in the proofs of Lemmas~\ref{l:A1}  and ~\ref{l:Zset} .

\end{proof}

\end{document}